\documentclass[11pt]{article}
\usepackage{amssymb}
\usepackage{amsthm}
\usepackage{amsmath}
\usepackage{latexsym}
\usepackage{graphics}
\usepackage{xspace}
\usepackage{psfrag}
\usepackage{epsfig}
\usepackage{pst-all}
\usepackage{algorithmic}

\newtheorem{theorem}{Theorem}[section]
\newtheorem{lemma}[theorem]{Lemma}

\newtheorem{proposition}[theorem]{Proposition}
\newtheorem{coro}[theorem]{Corollary}

\newtheorem{OpenProblem}[theorem]{Open Problem}
\newtheorem*{remark}{Remark}

\DeclareMathOperator\erf{erf}

\DeclareMathOperator{\sgn}{sgn}

\newcommand{\ER}{Erd\H{o}s-R\'{e}nyi }
\newcommand{\pr}{\mathbb{P}}
\newcommand{\E}{\mathbb{E}}
\newcommand{\R}{\mathbb{R}}
\newcommand{\Z}{\mathbb{Z}}
\newcommand{\G}{\mathbb{G}}

\newcommand{\T}{\mathbb{T}}

\newcommand{\distr}{\stackrel{d}{=}}
\newcommand{\Pois}{\ensuremath{\operatorname{Pois}}\xspace}
\newcommand{\Gncn}{\G(n,\lfloor cn\rfloor)}

\newcommand{\ignore}[1]{\relax}

\begin{document}

\title{\Large On the Max-Cut of Sparse Random Graphs
\thanks{  }}

\author{
{\sf David Gamarnik}\thanks{MIT; e-mail: {\tt gamarnik@mit.edu}.Research supported  by the NSF grants CMMI-1335155.}
\and
{\sf Quan Li}\thanks{MIT; e-mail: {\tt quanli@mit.edu}}
}

\date{}

\maketitle

\begin{abstract}
We consider the problem of estimating the size of a maximum cut (Max-Cut problem) in a random \ER graph on $n$ nodes and $\lfloor cn \rfloor$ edges. It is shown in Coppersmith et al. ~\cite{Coppersmith2004} that the size of the maximum cut in this graph normalized by the number of nodes belongs to the asymptotic region $[c/2+0.37613\sqrt{c},c/2+0.58870\sqrt{c}]$ with high probability (w.h.p.) as $n$ increases, for all sufficiently large $c$. The upper bound was obtained by application of the first moment method, and the lower bound was obtained by constructing algorithmically a cut which achieves the stated lower bound.

In this paper we improve both upper and lower bounds by introducing a novel bounding technique. Specifically, we establish that the size of the maximum cut normalized by the number of nodes belongs to the interval  $[c/2+0.47523\sqrt{c},c/2+0.55909\sqrt{c}]$ w.h.p. as $n$ increases, for all sufficiently large $c$. Instead of considering the expected number of cuts achieving a particular value as is done in the application of the first moment method, we observe that every maximum size cut satisfies a certain local optimality property, and  we compute the expected number of cuts with a given value satisfying this local optimality property. Estimating this expectation amounts to solving a rather involved two dimensional large deviations problem. We solve this underlying large deviation problem asymptotically as $c$ increases and use it to obtain an improved upper bound on the Max-Cut value. The lower bound is obtained by application of the second moment method, coupled with the same local optimality constraint,
and is shown to work up to the stated lower bound value $c/2+0.47523\sqrt{c}$. It is worth noting that both bounds are stronger than the ones obtained by standard first and second moment methods. 

Finally, we also obtain an improved lower bound of $1.36000n$ on the Max-Cut for the random cubic graph or any cubic graph with large girth, improving the previous
best bound of $1.33773n$.
\end{abstract}

\section{Introduction and Main Results}\label{section:Intro}
\subsection{Context and previous results}
The Max-Cut (finding the maximum cut) of a graph is the problem of splitting the nodes of a graph into two parts so as to maximize the number of edges between the two parts.
In the worst case the problem falls into the Max-SNP-hard complexity class which means that the optimal value cannot be approximated within a certain multiplicative error by a
polynomial time algorithm, unless P=NP. In this paper, however, we are concerned with the average case analysis of the Max-Cut problem.
Last decade we have seen a dramatic progress
improving our understanding of various randomly generated constraint satisfaction models such as the random K-SAT problem, the random XOR-SAT problem, proper coloring of a random graph,
independence ratio of a random graph, and many related problems~\cite{ding2013satisfiability}, \cite{coja2013upper}, \cite{coja2014asymptotic}.
These problems broadly fall into the class of so-called anti-ferromagnetic spin glass models, borrowing a terminology from statistical physics.
A great deal of progress was also achieved in studying ferromagnetic counterparts of these problems on random and more general
locally tree-like graphs~\cite{DemboMontanariIsing},\cite{DemboMontanariSurvey},\cite{DemboMontanariSun}.
At the same time the best known results for the Max-Cut problem (which falls into anti-ferromagnetic category)  were obtained in~\cite{Coppersmith2004} about a decade ago,
and have not been improved
ever since. In the aforementioned reference, upper and lower bounds are obtained on the Max-Cut value for the sparse random \ER graph.
Related results concerning the Max-K-SAT problem were considered in~\cite{achlioptas2005rigorous} and~\cite{achlioptas2007maximum}.
In this paper we improve both upper and lower bounds on the Max-Cut value
obtained in~\cite{Coppersmith2004}, using a new method based on applying local optimality property of maximum cuts
and solving an underlying two-dimensional large deviations problem.

Recall that an \ER graph $\G(n,m)$ is a random graph generated by selecting $m$ edges uniformly at random (without replacement) from all possible edges on $n$ vertices. The Max-Cut problem exhibits a phase transition at $2m/n=1$. Specifically, Coppersmith et al.~\cite{Coppersmith2004}  showed that the difference of $m$ and the MaxCut size jumps from $\Theta(1)$ to $\Theta(n)$ as $2m/n$ increases from below to above $1$. Furthermore, Daud\'{e}, Mart\'{i}nez, et al. \cite{daude2012max} established the distributional limit
of Max-Cut size in the scaling window $2m-n \ll n$.

Let $m = \lfloor cn \rfloor$ for some constant $c$. When $c$ is sufficiently large, which is the setting considered in this paper, both upper and lower bounds of the Max-Cut size are also obtained in~\cite{Coppersmith2004}.
To describe their result, let $MC_{n,c}$ denote the Max-Cut value in the \ER graph $\G(n,\lfloor cn \rfloor)$. Then there exists $\mathcal{MC}(c)$ such that
\begin{align}\label{eq:MaxCutLimit}
{MC_{n,c}\over n}\rightarrow \mathcal{MC}(c)
\end{align}
in probability as $n\rightarrow\infty$.
The existence of this limit is by no means obvious and itself was only recently established in~\cite{BayatiGamarnikTetali}. The fact that the actual value concentrates
around $\mathcal{MC}(c)$ with high probability follows directly by application of the Azuma's inequality.
In terms of $\mathcal{MC}(c)$, it was shown in~\cite{Coppersmith2004} that $\mathcal{MC}(c)\in [c/2+0.37613\sqrt{c}+o_c(\sqrt{c}),c/2+0.58870\sqrt{c}+o_c(\sqrt{c})]$, where
$o_c(\sqrt{c})$ denotes a function $f(c)$ satisfying $\lim_{c\rightarrow\infty}f(c)/\sqrt{c}=0$. From here on we use standard notations $o(\cdot),O(\cdot)$ and $\Theta(\cdot)$
with respect to $n\rightarrow\infty$.
When these order of magnitude notations are with respect to the regime $c\rightarrow\infty$, we use subscripts $o_c,O_c,\Theta_c$.
The upper bound was obtained by using a standard first moment method. Namely,
one computes the expected number of cuts achieving a certain cut size value. It was shown that when the size is at least $c/2+0.58870\sqrt{c}+o(\sqrt{c})$ the expectation
converges to zero exponentially fast, and thus the cuts of this size do not exist w.h.p. For the lower bound the authors constructed an algorithm where the nodes
were dynamically assigned to different parts of the cut based on the majority of the implied degrees. Since the degree of a node has approximately a Poisson distribution
with parameter $2c$, which for large $c$ is approximated by a Normal distribution with mean $2c$ and standard deviation $\sqrt{2c}$, the maximum of two such random
variables is approximately a maximum of two normally distributed random variables with the same distribution and has mean of order $\sqrt{c}$. This approach leads to a lower
bound $c/2+0.37613\sqrt{c}+o(\sqrt{c})$.  Coja-Oghlan and Moore \cite{coja2003max} generalized the similar ideas to Max $k$-Cut problem and proposed an approximation algorithm by using semidefinite relaxations of Max $k$-Cut. The approximation of a Poisson distribution by a Normal distribution when parameter of the Poisson distribution is large is also
instrumental in the analysis used in our paper.

\subsection{Our contribution}\label{subsection:Our Contribution}
In this paper we obtained improved upper and lower bounds on the Max-Cut value in \ER graph when the edge density $c$ diverges to infinity.
We now state our results precisely. Our bounds will be expressed in terms of solutions
to somewhat complicated equations which we introduce now.

We begin with equations involved in the upper bound on the value of the Max-Cut.
Consider the following system of two equations in variables $x$ and $\theta$
\begin{align}
\label{transce1}
-2x^2 + \theta^2 + \log \left( 1+ \erf (2x+\theta)  \right)&=0, \\
\label{transce2}
\theta + \sqrt{\frac{1}{\pi}} \frac{e^{-(2x+\theta)^2}}{1+\erf(2x+\theta)}&=0,
\end{align}
where $\erf(\cdot)$ is the Gaussian error function, defined as
\begin{align}
\label{error_func}
\erf(x)=\frac{2}{\sqrt{\pi}} \int_0^x \exp(-t^2) \, dt.
\end{align}
In particular, $\erf(x/\sqrt{2})=\sgn(x)\mathbb{P}[|Z|\leq |x|]$ when $Z$ is the standard normal random variable. We denote by $w_1(\theta,x)$ and $w_2(\theta,x)$
the functions appearing on the left-hand of (\ref{transce1}) and (\ref{transce2}) respectively.

\begin{lemma}\label{lemma:UniqueSolution}
For every $x$ in the range
\begin{align}\label{eq:xupperrange}
x\in [0.37613, 0.58870]
\end{align}
the equation $w_2(x,\theta)=0$ in $\theta$ has a unique solution. Furthermore,
the system (\ref{transce1}) and  (\ref{transce2}) has a unique solution in the same region (\ref{eq:xupperrange}).
Numerically,  this unique solution is $x_u=0.55909..$ and $\theta_u=-0.11079..~$.
\end{lemma}
The proof of this lemma is given at the end of Section \ref{section:UpperBound}. The interval $[0.37613, 0.58870]$ appearing above is the upper and lower bound values derived in~\cite{Coppersmith2004}. We use it as a convenient
guarantee that the ``true'' value of $x$ has to belong to this range. We denote by $w(x)$ the univariate function $w_1(x,\theta(x))$, where
$\theta(x)$ is the unique solution of $w_2(x,\theta)=0$:
\begin{align}\label{eq:wx}
w(x)=w_1(x,\theta(x)).
\end{align}
We now introduce equations involved in deriving the lower bound on the Max-Cut value. Given $x$ and $\beta\in(0,1/2)$,
consider the following system of three equations in variables $t$, $\theta_1$ and $\theta_2$
\begin{align}
\label{partF_the1}
 &\theta_1 Q \left(\theta_1,\sqrt{\frac{1/2-\beta}{\beta}},\frac{t}{\beta^{3/2}}  \right)+\int_0^{\infty}
 \exp \left(-{z^2\over 2}+{1\over 2}\left(\sqrt{\frac{1/2-\beta}{\beta}}z-\theta_1-\frac{t}{\beta^{3/2}} \right)^2  \right) dz=0,  \\
\label{partF_the2}
 &\theta_2 Q \left(\theta_2,\sqrt{\frac{\beta}{1/2-\beta}},\frac{x-t}{(1/2-\beta)^{3/2}}  \right)+\int_0^{\infty}
 \exp \left(-{z^2\over 2}+{1\over 2}\left(\sqrt{\frac{\beta}{1/2-\beta}}z-\theta_2-\frac{x-t}{(1/2-\beta)^{3/2}} \right)^2  \right) dz=0, \\
\label{partF_t}
&-\frac{t}{\beta^2}+\frac{x-t}{(1/2-\beta)^2}-2\frac{\theta_1}{\beta^{1/2}}+2\frac{\theta_2}{(1/2-\beta)^{1/2}}=0,
\end{align}
where
\begin{align}
\label{Q_exp}
Q(\theta,a_1,a_2)=\int_{0}^{\infty} \int_{a_1 z_2}^{\infty} \exp(-((z_1-\theta-a_2)^2+z_2^2)/2) dz_1 dz_2.
\end{align}
\begin{lemma}\label{lemma:UniqueSolution2}
For every $x$ satisfying (\ref{eq:xupperrange}) and  $\beta \in (0,1/2)$, the system (\ref{partF_the1}), (\ref{partF_the2}) and (\ref{partF_t}) has a unique solution.
\end{lemma}
This lemma follows from Lemma \ref{lemma:Uniquesaddlepoint} whose proof is given in Section \ref{section:SystemEquation}. Given $x$ and $\beta\in(0,1/2)$, denote the unique solution to (\ref{partF_the1}), (\ref{partF_the2}) and (\ref{partF_t}) by
$\theta_1^*(x,\beta)$, $\theta_2^*(x,\beta)$ and $t^*(x,\beta)$. We introduce the following functions
\begin{align}
P(\theta,a_1,a_2)= &\frac{1}{\pi} \exp(\theta^2/2)\int_{0}^{\infty} \int_{a_1 z_2}^{\infty} \exp(-((z_1-\theta-a_2)^2+z_2^2)/2) dz_1 dz_2,  \label{eq:Pfunction} \\
W(x,\beta)=&-2\beta\log\beta-2(1/2-\beta)\log(1/2-\beta) \notag\\
&-{1\over 2}{t^*(x,\beta)^2\over \beta^2}-{1\over 2}{(x-t^*(x,\beta))^2\over (1/2-\beta)^2}
+2\beta
\log P\left(\theta_1^*(x,\beta),\sqrt{1/2-\beta\over \beta},
{t^*(x,\beta)\over \beta^{3/2}}\right) \notag\\
&+2(1/2-\beta)
\log P\left(\theta_2^*(x,\beta),\sqrt{\beta\over 1/2-\beta},
{x-t^*(x,\beta)\over (1/2-\beta)^{3/2}}\right) , \label{eq:Wxbeta}
\end{align}
where the first function is defined for all $\theta,a_1,a_2\in\R$ and the second function is defined for all $x$ in the range (\ref{eq:xupperrange})
and $\beta\in(0,1/2)$.
Let
\begin{align}
\label{lower_bound_x}{}
x_l=\sup \left\{x \in [0.37613, x_u) :{} \sup_{\beta\in (0, 1/2)}W(x,\beta)=2w(x) \right \}.
\end{align}
The functions $\sup_{\beta \in (0, 1/2)} W(x, \beta)$ and 
$2\omega(x)$ for $x \in [0.44, 0.56]$ are given in Figure \ref{bifurcate}, which shows that there is a bifurcation between $\sup_{\beta \in (0, 1/2)} W(x, \beta)$ and $2\omega(x)$ within $x \in [0.44, 0.56]$. 
We find $x_l = 0.47523..$ in Section \ref{section:SystemEquation}, assuming the validity of a numerical search procedure on finding a solution to a set of nonlinear equations. The plots of $W(x, \beta)$ for $x = 0.47523$ and $x = 0.5$ are also given in Figures \ref{SMMM} and \ref{SMMM2}. 

\begin{figure}[!htb]
\vspace{0pc}
\centering
\includegraphics[width=5 in]{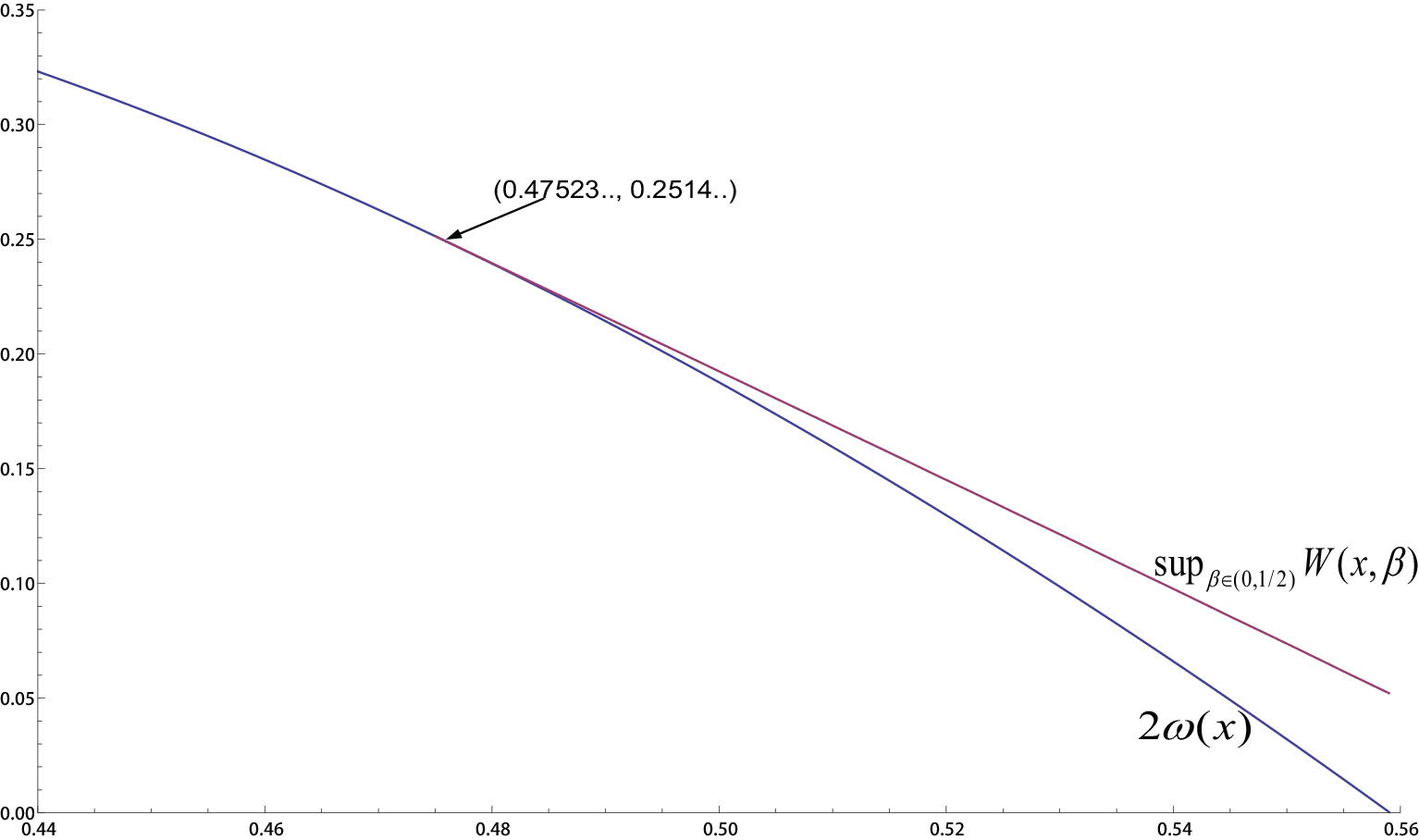}
\caption{$\sup_{\beta\in (0, 1/2)}W(x,\beta)$ and $2w(x)$ for $x \in [0.44, 0.56]$.}
\label{bifurcate}
\end{figure}

The values $x_u$ and $x_l$ give the new upper and lower bounds on the Max-Cut value as stated in the main result of this paper below.
\begin{theorem}\label{main_theorem}
Let $\mathcal{MC}(c)$ be defined as in (\ref{eq:MaxCutLimit}). Then
\begin{align}
\label{mc_Gc}
\mathcal{MC}(c) \in [c/2+x_l \sqrt{c}+o_c(\sqrt{c}), c/2+x_u \sqrt{c}+o_c(\sqrt{c})],
\end{align}
where $x_l$ is defined in (\ref{lower_bound_x}) and $x_u$ is defined as in Lemma~\ref{lemma:UniqueSolution}.
\end{theorem}
{\remark After this paper was completed and posted on arxive.org, a follow-up paper by Dembo, Montanari and Sen \cite{Dembo2015} resolved our open problem \ref{openproblem1}, namely, closed the gap between the constants $x_l$ and $x_u$, using the interpolation technique on the Sherrington-Kirkpatrick model.}

This theorem can be easily extended to another variant of \ER random graph $\G(n, p = 2c/n)$, which is defined by putting every one of the $n(n-1)/2$ potential edges into the graph with
probability $2c/n$, independently for all edges.
Since the number of edges in $\G(n,p=2c/n)$ is tightly concentrated around $\lfloor cn \rfloor$ with fluctuation bounded by $O(n^{1/2+\epsilon})$ for any $\epsilon>0$ w.h.p., the Max-Cut size bounds derived from $\G(n, m = \lfloor cn\rfloor )$ also apply to $\G(n,p=2c/n)$.

Our result has immediate ramification to a very related problem of estimating the energy of a ground state of an anti-ferromagnetic Ising model at zero temperature.
Given an arbitrary undirected graph $G$ with node set $V$ and edge set $E$ and a real value $\beta$, the Ising model corresponds to a Gibbs distribution
on the state space $\{-1,1\}^{|V|}$ defined by
\begin{align*}
\pr(\sigma)=Z^{-1}\exp(-\beta\sum_{(u,v)\in E} \sigma_u\sigma_v),
\end{align*}
for every $\sigma=(\sigma_u, u\in V)\in \{-1,1\}^{|V|}$, where $Z=\sum_\sigma \exp(-\beta\sum_{(u,v)\in E} \sigma_u\sigma_v)$ is the normalizing partition function.
The case $\beta>0$ corresponds to the anti-ferromagnetic Ising model, and the ground state $\sigma^*$ is any state which minimizes the energy
functional $\sum_{(u,v)\in E} \sigma_u\sigma_v$, namely the one maximising the Gibbs likelihood. There is an obvious simple one-to-one relationship between
energy of ground states of an anti-ferromagnetic Ising model and Max-Cut problem. Denoting by $H(G)$ the energy of a ground state, we have
$H(G)=|E|-2MC(G)$, where $MC(G)$ denotes the Max-Cut value of the graph $G$. Denote by $\mathcal{I}(c)$ the limit of the ground state energy normalized
by $n$, as $n\rightarrow\infty$. The existence of this limit follows from the existence of the corresponding limit $\mathcal{MC}(c)$ for every $c$.
As an immediate implication of Theorem~\ref{main_theorem}, since the number of edges in $\G(n, \lfloor cn \rfloor)$ is $\lfloor cn\rfloor$ we obtain
\begin{coro}\label{main_theorem_Ising}
The following bounds hold
\begin{align}
\label{mc_GcIsing}
\mathcal{I}(c) \in [-2x_u \sqrt{c}+o_c(\sqrt{c}), -2x_l \sqrt{c}+o_c(\sqrt{c})],
\end{align}
where $x_l$ and $x_u$ have the same values as in Theorem~\ref{main_theorem}.
\end{coro}

The main novel technique underlying the bounds presented in Theorem~\ref{main_theorem} is based on the local optimality property of the maximum cuts. Specifically, given an arbitrary
graph $G$ with a node set $V$ and edge set $E$, let $V_1,V_2$ be any node partition which maximizes $E(V_1,V_2)$, where $E(A,B)$ denotes the number of edges between disjoint node
sets $A\subset V$ and $B\subset V$. Namely, $V_1,V_2$ achieve the maximum cut value. For every $v\in V_1$, let $N_1(v)$ and $N_2(v)$ denote the neighbors of node $v$ in parts
$V_1$ and $V_2$ respectively. Optimality of $(V_1,V_2)$ implies that
 $|N_1(v)|\le |N_2(v)|$, as otherwise a higher cut value can be obtained by assigning $v$ to $V_2$ instead of $V_1$.
A similar observation holds for every node $v\in V_2$.
We say that a (not necessarily optimal) cut (node partition) $V_1,V_2$
satisfies the local optimality constraint if this property holds for every node $v$ in $V_1$ and $V_2$. Clearly every optimal cut
satisfies the local optimality constraint. Our main approach is based on computing the expected number of cuts which satisfy the local optimality
constraint and which achieve a certain cut value $\lfloor zn \rfloor$ for a constant $z$. Computing this expectation is an involved task and amounts to solving a certain two-dimensional large deviations problem.
The nature of this problem can be described as follows. Consider the random multi-graph as the configuration model generalized to \ER graph (later we will explain it in details in the next Section). The joint distribution of degrees of nodes in this random multi-graph can be described by the joint distribution arising from the balls
into bins problem. Specifically, for an even $n$, given a cut $V_1,V_2$ of a graph of equal size $|V_1|=|V_2|=n/2$ (later we will establish that this case determines the normalized exponent of the expected number of cuts which satisfy the local optimality condition by $n$ as $n \rightarrow \infty$),
conditioned to have value $\lfloor zn \rfloor$, such that the remaining parts $V_1$ and $V_2$ have the number of internal edges equal to $\lfloor z_1n \rfloor$ and $\lfloor z_2n \rfloor$ for two constants $z_1$ and $z_2$ respectively, with $\lfloor zn \rfloor + \lfloor z_1n \rfloor+ \lfloor z_2n \rfloor= \lfloor cn \rfloor$,
the joint distribution of the number of neighbors
of nodes of $V_1$ in part $V_2$ is described as the joint distribution arising from putting $\lfloor zn \rfloor$ balls into $n/2$ bins uniformly at random. Similarly,
the joint distribution of the number of neighbors of nodes of $V_j$  who also belong to $V_j$
is also described as the joint distribution arising from putting $\lfloor 2z_jn \rfloor$ balls
into $n/2$ bins uniformly at random, independently from the first process and from the other part.
Let the first $\lfloor zn \rfloor$ balls be colored blue, and the balls corresponding to the $\lfloor z_jn \rfloor$ edges be colored red for $j=1,2$.
Then the local optimality constraint means that in each bin the number of red balls does not exceed the number of blue balls. Achieving a particular cut value
$\lfloor zn \rfloor$ amounts to saying that the total number of blue balls equals $\lfloor zn \rfloor$. Both events are of large deviations type and computing the likelihood of this rare event amounts
to solving a two-dimensional large deviations problem. While solving this problem for a fixed $c$ appears to be intractable, it
can be solved asymptotically when $c$ is large since in this case the distribution of balls in bins
is well approximated by a normal distribution. As a result the large deviations rate function can be solved by integration over Gaussian distribution. This approach
leads to an upper bound stated in our main theorem.

To obtain the lower bound we consider the second moment of the number of cuts achieving value $z$ satisfying the local optimality constraint. The idea of the approach is
very similar as in the case of the upper bound, but details are more involved since we consider now pairs of cuts. We use the second moment method to obtain
a lower bound on the probability of existence of a cut with a particular value. This lower bound still is exponentially small. Our last step is to use an exponential
concentration of the Max-Cut value around its expectation in order to argue the existence of a cut with a stated value. The last step is similar to the one
used in earlier papers, such as Frieze~\cite{FriezeIndependentSet}.

Ideas somewhat similar to our local optimality condition, appear in a different context of random K-SAT problem. There the
single-flip satisfying truth assignment is used to obtain the upper bounds on the $3$-satisfiability threshold in~\cite{dubois2000typical}, and \cite{diaz2008new}.
The idea in these works was to count the expected number of those satisfying truth assignments which are local maxima in terms of a lexicographic ordering.
While the idea of using local optimality property in these papers and in our paper is somewhat similar, the details of the analysis differ substantially.

Our last result concerns maximum cut in cubic (namely $3$-regular) graphs.
Here the best known bound follows from a recent result by Lyons~\cite{lyons2014factors} who proves existence
of a cut with an asymptotic value at least $1.33773n$ in an arbitrary sequence of cubic connected graphs, whose girth (size of a smallest cycle) diverges to infinity. It is worth noting that Lyons' result also applys to maximum bisection for which to our best knowledge his result is still the state-of-the-art.  
Our improved bound is based on a simple argument taking advantage of a recent result by  Cs\'{o}ka et al.~\cite{csoka2014invariant} regarding the size of a largest bi-partite
subgraph of a cubic graph with large girth. We obtain
\begin{theorem}
\label{cgraph_cut}
Let $\G_n$ be an arbitrary sequence of $n$-node cubic connected graphs with girth diverging to infinity. For these graphs
\begin{align*}
\liminf_n{MC_{n,c}\over n}\ge 1.36000..~.
\end{align*}
\end{theorem}
Note that while the girth of the random $n$-node cubic graph (a graph generated uniformly at random from the set of all $3$-regular $n$-node graphs)
does not necessarily diverge to infinity, this graph does have mostly a locally tree-like structure and the results which regard ``global'' structure such as Max-Cut obtained from the regular graphs with diverging girth apply to these graphs as well, see for
one example where such an argument is developed~\cite{BandyopadhyayGamarnikCounting}. Specifically, one can use the construction described on page 22, 
Subsection 4.4 of the aforementioned paper. In this paper a simple procedure is described consisting of ``blowing up'' a portion of the random graph containing
small cycles into a part which does not contain cycles of any fixed length $g$. Since it affects only a constant size portion of the graph it does not affect
the limiting value of a maximum cut.

The remainder of this paper is organized as follows. In the next section we provide some preliminary technical results regarding the balls into bins model.
In the same section we state and prove the so-called local large deviations results for lattice based random variables. These results serve as a basis for computing
the first and second moments of the number of cuts satisfying the local optimality constraints. Section~\ref{section:UpperBound} is devoted to establishing the upper
bound part of our main result, Theorem~\ref{main_theorem}, using the first moment method.
In Section~\ref{section:LowerBound} we derive an optimization problem the solution of which describes the asymptotics of second moment.
In Section~\ref{section:SystemEquation} this optimization problem is reduced to a system of equations, the unique solution of which is used to obtain the lower bound
on the maximum cut value. Most of the ideas are based on the same techniques as the ones used for the upper bound part, but the details are very
lengthy and far more involved. Section~\ref{section:proof_MaxCut_cubic} is dedicated to the proof of Theorem \ref{cgraph_cut}. The numerical answers appearing in the statement and the proofs of our result are based on computer assisted computation and thus
our results should be qualified as computer assisted. In the last Section we conclude with several open problems.

\section{Preliminary results. Random multi-graphs, the Balls into Bins model and the Local Large Deviations bounds}\label{section:preliminary}
Our random graph model $\G(n,m)$ model is obtained by selecting $m$ out of $n(n-1)/2$ edges uniformly at random without replacement. The analysis
below is significantly simplified by switching to a more tractable random multi-graph model generated from the configuration model
where edge repetition and loops are allowed. Then
we use a fairly standard observation that this change does not impact the asymptotic value of the Max-Cut. Thus consider the set of $n(n+1)/2$
edges on $n$ nodes, which now include $n$ loops and suppose we select $m$ edges uniformly at random with replacement. Equivalently, one
can think of this as an experiment of throwing  $2m$ balls (also commonly called clones)  into $n$ bins (nodes) of the graph uniformly at random,
and then creating a random $m$-matching between the $2m$ balls. An edge between node $i$ and $j$ is formed
if and only if there exist two balls thrown into bins $i$ and $j$ which are connected in the matching. In particular loops and parallel edges are allowed,
though it is easy to check that when $m=O(n)$, with probability bounded away from zero as $n\rightarrow\infty$, the number of loops and parallel edges is zero.
Conditional on this event, the resulting graph is $\G(n,m)$. Since all the results obtained in this paper hold w.h.p., we now assume from this point on that
$\G(n,\lfloor cn\rfloor)$ stands for the random multi-graph model described above.

In order to implement the local optimality condition for Max-Cut, we first introduce two relevant lemmas regarding the variant of the so-called occupancy (Balls into Bins) problem.
 In order to decouple the distribution of the number of balls each bin receives, we need the following ``Poissonization lemma" \cite{durrett2010probability,coja2013upper}.
\begin{lemma}
{\upshape \cite[Exercise 3.6.13]{durrett2010probability}; \cite[Corollary 2.4]{coja2013upper}}
\label{poissonization}
Consider an experiment where $\mu \in \mathbb{N}$ balls are thrown independently and uniformly at random (u.a.r.) into $n$ bins.
Let $E_i$ be the number of balls in bin $i \in [n]\triangleq \{1,2,\ldots,n\}$. Let $\lambda=\mu/n>0$ and $(B_i)_{i \in [n]}$ be a family of independent Poisson variables with the same mean $\lambda$. Then for any sequence $(t_i)_{i \in [n]}$ of non-negative integers such that $\sum_{i=1}^{n} t_i=\mu$ we have
\begin{align}
\label{Poi_approx1}
\mathbb{P}[ E_i=t_i, 1\leq i\leq n] &=\mathbb{P}[ B_i=t_i, 1\leq i\leq n| \sum_{i=1}^{n} B_i=\mu]
                                      =\Theta_{\mu}(\sqrt{\mu}) \mathbb{P}[ B_i=t_i, 1\leq i\leq n].
\end{align}
\end{lemma}
Here the standard order of magnitude notation $\Theta_{\mu}(\sqrt{\mu})$ denotes a non-negative function $f(\mu)$ such that
\begin{align*}
0<\liminf_{\mu\rightarrow\infty}{f(\mu)\over \sqrt{\mu}}\le \limsup_{\mu\rightarrow\infty}{f(\mu)\over \sqrt{\mu}}<\infty.
\end{align*}
Consider now an experiment of throwing balls into bins twice. First $\mu_1$ balls are thrown independently and u.a.r. into $n$ bins. Denote the number of balls in bin $i$ by $E_i$.
Next, the bins are reset empty and another $\mu_2$ balls are thrown  u.a.r. into $n$ bins independently for all bins and independently from the first experiment.
Denote the number of balls in bin $i$ by $F_i$. Correspondingly, let $B_i$, $1 \leq i \leq n$, and $C_i$, $1 \leq i \leq n$, be two families of independent Poisson variables with means $\lambda_1=\mu_1/n$ and $\lambda_2=\mu_2/n$, respectively. We rely on Lemma \ref{poissonization} to evaluate the probability
\begin{align}\label{eq:Kmu_mubar}
K(n,\mu_1,\mu_2)\triangleq \mathbb{P}[ E_i \geq F_i , 1\leq i\leq n].
\end{align}

\begin{lemma}
\label{lemma_twoPo}
The following holds
\begin{align}
K(n,\mu_1,\mu_2)
& = \Theta_{\mu_1}(\sqrt{\mu_1}) \Theta_{\mu_2}(\sqrt{\mu_2})
\mathbb{P}\left[\sum_{i=1}^{n} B_i=\mu_1, \sum_{i=1}^{n} C_i=\mu_2 \Bigm\vert B_i \geq C_i, 1\leq i\leq n \right]
(\mathbb{P}[B_1 \geq C_1 ])^{n}. \label{Pos_approx}
\end{align}
\end{lemma}
\begin{proof}[Proof of Lemma~\ref{lemma_twoPo}]
Let
\begin{align*}
S(\mu_1,\mu_2)=\{((t_i, s_i))_{1\leq i \leq n} \in (\mathbb{Z}_{\geq 0})^{2n}: t_i \geq s_i, 1\leq i \leq n; \sum_{i=1}^{n}t_i=\mu_1; \sum_{i=1}^{n}s_i=\mu_2\}.
\end{align*}
We have
\begin{align*}
 \mathbb{P}[E_i \geq F_i, & 1\leq i\leq n]= \\
=& \sum_{S(\mu_1,\mu_2)} \mathbb{P}[E_i=t_i, 1\leq i\leq n]\mathbb{P}[F_i=s_i, 1\leq i\leq n] \nonumber \\
=& \sum_{S(\mu_1,\mu_2)} \Theta_{\mu_1}(\sqrt{\mu_1}) \Theta_{\mu_2}(\sqrt{\mu_2}) \mathbb{P}[B_i=t_i, 1\leq i\leq n]
\mathbb{P}[C_i=s_i, 1\leq i\leq n] \nonumber \\
=& \Theta_{\mu_1}(\sqrt{\mu_1}) \Theta_{\mu_2} (\sqrt{\mu_2})
\mathbb{P}[B_i \geq C_i, 1\leq i\leq n, \sum_{i=1}^{n}B_i=\mu_1, \sum_{i=1}^{n} C_i=\mu_2 ] \nonumber \\
																								=& \Theta_{\mu_1}(\sqrt{\mu_1}) \Theta_{\mu_2}(\sqrt{\mu_2})
\mathbb{P}\left[\sum_{i=1}^{n} B_i=\mu_1, \sum_{i=1}^{n} C_i=\mu_2 \Bigm\vert B_i \geq C_i, 1\leq i\leq n \right]
\mathbb{P}[B_i \geq C_i, 1\leq i\leq n] \nonumber
\end{align*}
By the independence of $B_i$ and $C_i$, $1\leq i \leq n$ ,we have (\ref{Pos_approx}).
\end{proof}

In order to compute the conditional probability in (\ref{Pos_approx}), we need to rely on  multivariate local limit theorems for
large deviations~\cite{richter1958multi},\cite{chaganty1985large},\cite{chaganty1986multidimensional}.
The classical large deviations theory provides tight estimates of the exponent $\gamma$ appearing when calculating the rare events of the form
$\pr(X_n>nx)\approx \exp(-\gamma xn)$. The local large deviations theory instead provides estimates of the form $\pr(X_n=nx)\approx \exp(-\gamma xn)$,
where usually the same exponent $\gamma$ governs the large deviations rate. Naturally, the local case is restricted to cases when values $nx$ belong to
the range of random variables $X_n$.

\ignore{ which were were derived in
Richter~\cite{Richter1958} and then improved by Chaganty and Sethuraman~\cite{Chaganty1985, Chaganty1986}. Since the theorems they developed are mainly for the general case of arbitrary sequences of random variables with assumptions which not easy to check, here we will typically rely on the local Central Limit Theorems to derive a local theorem for large deviations tailored for the application in this paper.}

Thus let $\{e_1,\dots,e_d\}$ be an orthonormal basis of $\mathbb{R}^d$ where $e_i$ is the unit vector of $0$'s except for $1$ in the $i$th position. Let $X_1, X_2, \dots$ be i.i.d. random vectors in $\mathbb{R}^d$ with mean equal to vector $\mathbf{0}$ and finite second moment. Furthermore, suppose the covariance matrix $\Sigma$ is non-singular,
and the distribution of $X_i$ is supported on a lattice
with parameters $b\in\R^d, h_i\in \R, 1\le i\le d$. Namely,
\begin{align*}
\mathbb{P}[\exists~z_1,\ldots,z_d\in\Z: X_i =b + \sum_{i=1}^d h_i e_i z_i ]=1,
\end{align*}
and this is the smallest (in set inclusion sense) lattice with this property.
If $S_n=X_1+\dots+X_n$, then $S_n$ is of the form $nb+\sum_{i=1}^dh_i e_i z_i$ for some $z_i\in \Z, 1\le i\le d$. Let
\begin{align*}
p_n(x)=\mathbb{P}[S_n/\sqrt{n}=x].
\end{align*}
This probability is positive only when
\begin{align*}
x\in \mathcal{L}_n\triangleq\{(nb+\sum_{i=1}^dh_i e_i z_i)/\sqrt{n}, z_i \in  \mathbb{Z}, 1\le i\le d \}.
\end{align*}
Let
\begin{align*}
p(x)=\frac{1}{(2\pi)^{\frac{d}{2}}\sqrt{\lvert \Sigma \rvert}} \exp\left(-\frac{1}{2}x^T \Sigma^{-1}x \right) \quad \text{for} \; x\in \mathbb{R}^d.
\end{align*}
The following local Central Limit Theorem can be found as  Theorem 3.5.2 in \cite{durrett2010probability}.
\begin{theorem}
\label{LCLT}
Under the hypotheses above, as $n\rightarrow \infty$,
\begin{align}
\sup_{x\in \mathcal{L}_n} \left\lvert \frac{n^{\frac{d}{2}}}{\prod_{i=1}^d h_i}p_n(x)-p(x) \right\rvert \rightarrow 0.
\end{align}
\end{theorem}
Based on this result, we use the change-of-measure technique to obtain the following local limit theorem for large deviations, as in the proof of Cram\'er's Theorem and Sanov's Theorem in \cite{demzei98}. 
Let $M(\theta)=\mathbb{E}[e^{\langle \theta,X_1 \rangle}]$ be the moment generating function of $X_1$. Let
$\Lambda(\theta) \triangleq \log M(\theta)$.  Let also $\mathcal{D}_{\Lambda} \triangleq \{\theta \in \mathbb{R}^d: \Lambda(\theta)<\infty\}$ denote the domain
of the moment generating function of $X_1$. It is known~\cite{demzei98} that if $\mathcal{D}_{\Lambda}=\mathbb{R}^d$ then for every $y\in \R^d$ there exists a unique
$\theta^*\in\R^d$ such that $y=\nabla \Lambda(\theta^*)$. It is the unique $\theta$ which achieves the large deviation rate at $y$, namely
$\langle \theta^*,y \rangle-\Lambda(\theta^*)=\sup_\theta(\langle \theta,y \rangle-\Lambda(\theta))$.
\begin{theorem}
\label{LCLTLDP} Suppose $\mathcal{D}_{\Lambda}=\mathbb{R}^d$. Suppose  $y\in\R^d$ is such that $\sqrt{n}y\in\mathcal{L}_n$ for
all sufficiently large $n\in\Z_+$.
Let $\theta^*$ be defined uniquely by $y=\nabla \Lambda(\theta^*)$. Then,
\begin{align}
\lim_{n\rightarrow \infty} \frac{1}{n} \log \mathbb{P}[S_n/n=y] &=-\langle \theta^{*},y \rangle+\Lambda(\theta^*). \label{LLPD1}
\end{align}
\end{theorem}
The proof is obtained by combining a standard change of measure technique in the theory of large deviations with the local Central Limit Theorem~\ref{LCLT}.
We include the proof for completeness.

\begin{proof}[Proof of Theorem~\ref{LCLTLDP}]
Let $\mu$ be the probability measure associated with $X_1$ and $\mu_n$ be the probability measure associated with $S_n/n$. Using $y=\nabla \Lambda(\theta^{*})$, define a new probability measure $\tilde{\mu}$ with the same support as $X_1$  in terms of $\mu$ as follows:
\begin{align}
\label{measure_change}
\frac{d\tilde{\mu}}{d\mu}(z)=e^{\langle \theta^{*}, z \rangle-\Lambda(\theta^{*})},
\end{align}
It is easy to see that it is a probability measure by observing
$$
\int_{\mathbb{R}^d} d \tilde{\mu} = \int_{\mathbb{R}^d} e^{\langle \theta^{*}, z \rangle-\Lambda(\theta^{*})} d \mu=\frac{1}{M(\theta^{*})}\int_{\mathbb{R}^d} e^{\langle \theta^{*}, z \rangle}  d \mu=1.
$$

Let $\tilde{\mu}_n$ be the associated probability measure of $\tilde{S}_n=(\tilde{X}_1+\dots+\tilde{X}_n)/n$ where $\tilde{X}_i$, $1\leq i \leq n$, are i.i.d. random vectors with probability measure $\tilde{\mu}$. Then we have
\begin{align}
\mathbb{P}[S_n/n=y]=& \mu_n(\{y\})= \int_{\sum_{i=1}^n z_i=yn} \mu(dz_i)\nonumber \\
                           =& \int_{\sum_{i=1}^n z_i=yn}
                           e^{-\sum_{i=1}^n \langle \theta^{*}, z_i \rangle+n\Lambda(\theta^{*})}  \tilde{\mu}(dz_i)\nonumber \\
													 =& e^{-n\langle \theta^{*}, y \rangle + n \Lambda(\theta^{*}) } \tilde{\mu}_n(\{y\}).
\end{align}
Then we have
\begin{align}
\label{n_log_LPD}
 \frac{1}{n} \log \mathbb{P}[S_n/n=y]=-\langle \theta^*, y \rangle + \Lambda(\theta^{*}) + \frac{1}{n} \log \tilde{\mu}_n(\{y\}).
\end{align}
By the choice of $\theta^{*}$, we have
\begin{align}
\mathbb{E}_{\tilde{\mu}}[\tilde{X}_1]=\frac{1}{M(\theta^{*})}\int_{\mathbb{R}^d} z e^{\langle \theta^{*}, z \rangle}  d \mu = \nabla \Lambda(\theta^{*})=y,
\end{align}
where we used $\nabla M(\theta^{*})= \mathbb{E}[X_1 e^{\langle \theta^{*}, X_1 \rangle}] $, namely, the order of differentiation and expectation operators can be changed,
see for example Lemma 2.2.5 (c) in~\cite{demzei98}.
Furthermore, the probability measures $\tilde{\mu}$ defined in (\ref{measure_change}) has moments of all orders.
Then $\tilde X_1-y$ is a zero mean random vector with finite moments of all orders. Now since the lattice supporting $X_1$ and $\tilde X_1$ is the same,
the lattice supporting $\tilde X_i-y$ is described by parameters $b-y, h_i, 1\le i\le d$. Namely the same $h_i$ and $b$ replaced by $b-y$. Now the assumption
$\sqrt{n}y\in\mathcal{L}_n$ implies that $ny$ is of the form $nb+\sum_i h_ie_iz_i$, implying $n(b-y)+\sum_i h_ie_iz_i=0$. We conclude $0$ belongs to the
set $\mathcal{L}_n$ with $b-y$ replacing $b$.
Hence, Theorem \ref{LCLT} can be used to estimate the last term in (\ref{n_log_LPD}), which gives
\begin{align}
\lim_{n\rightarrow \infty}\frac{1}{n}\log \tilde{\mu}_n(\{y\})=0.
\end{align}

\end{proof}

\section{Upper bound. The first moment method}\label{section:UpperBound}

In this section, we establish the upper bound part of Theorem~\ref{main_theorem} using the first moment method.
We will prove the upper bound of Max-Cut size on $\Gncn$ by counting the expected number of cuts with a given cardinality,
satisfying the local optimality condition. For a constant $z$, let $X(\lfloor zn \rfloor)$ be the number of cuts $V_1,V_2=V(\Gncn)\setminus V_1$
of $\Gncn$ of size $\lfloor zn \rfloor$ in $\Gncn$ which satisfy the local optimality condition.
By results in~\cite{Coppersmith2004}, since we already know that the maximum cut size normalized by $n$ is
$c/2+\Theta_c(\sqrt{c})$ w.h.p., then for convenience we rescale $z$ by letting
$z=c/2+x\sqrt{c}$ for a positive real value $x$. According to the results from~\cite{Coppersmith2004} we know that we can limit ourselves to values $x\in[0.37613,0.58870]$.
\begin{proposition}\label{prop:comp_EbarX}
For every $x$ in (\ref{eq:xupperrange}) there exists a unique solution $\theta(x)$ of (\ref{transce2}).
Furthermore, for every $x$ in this range
\begin{align}
\label{EtildeEX}
\lim_{n\rightarrow \infty} \frac{1}{n} \log \mathbb{E}\left[X \left( \lfloor \left(c/2+x\sqrt{c} \right)n \rfloor \right) \right]=w(x) + o_c(1),
\end{align}
where $w(x)$ is defined in (\ref{eq:wx}).
\end{proposition}
We now show this result implies the upper bound part of Theorem~\ref{main_theorem}. We will establish later in the proof of Lemma~\ref{lemma:UniqueSolution} that
that $w(x)$ is a strictly decreasing function. Thus the expression above is positive (negative) if $x$ is smaller (larger) than the solution value $x_u$, for sufficiently large $c$.
The result is then obtained by Markov inequality. The remainder of the section is devoted to the proof Proposition~\ref{prop:comp_EbarX}.

We begin with some preliminary results.
Given $0\le \alpha\le 1/2$, consider a cut with two separate vertex subsets $V_1$ and $V_2$ with sizes $(1/2+\alpha)n$ and $(1/2-\alpha)n$,
respectively. There are $n \choose {(1/2+\alpha)n}$ such cuts. Given a positive even integer $m$, let $F(m)=\frac{m!}{(m/2)! 2^{m/2}}$. This is the number of
perfect matchings on a set of $m$ nodes. From this point on in all of our computations we will ignore the roundings $\lfloor\cdot\rfloor$ as they only contribute an extra term $O(1/n)$ to $\frac{1}{n} \log \mathbb{E}\left[X (zn) \right]$.
\begin{lemma}
The expected number of cuts of size $zn$ which satisfy the local optimality condition is
\begin{align}
\mathbb{E}[X(zn)]=\sum_{z_1,z_2,\alpha}J_1J_2, \label{halfalpha}
\end{align}
where the sum runs over all non-negative $z_1,z_2$ such that $z_1n,z_2n$ are integers and $z_1+z_2=c-z$,  over all
$\alpha\in [0,1/2]$ such that $(1/2+\alpha)n,(1/2-\alpha)n$ are integers, and
\begin{align}
J_1&={n \choose {(1/2+\alpha)n}} {{2cn}  \choose {{2z_1n},  {2z_2n}, {zn}, {zn}}} ((1/2+\alpha)n)^{(2z_1+z)n} ((1/2-\alpha)n)^{(2z_2+z)n} (zn)!\times  \\
&\times F(2z_1n)F(2z_2n)n^{-2cn} (F(2cn))^{-1},
\nonumber
\end{align}
and
\begin{align*}
&J_2=K((1/2+\alpha)n , zn, 2z_1 n) K((1/2-\alpha)n , zn, 2z_2 n), \notag
\end{align*}
where $K$ is defined by (\ref{eq:Kmu_mubar}).
\end{lemma}

As the cut size $zn$ increases, the number of cuts of such a size is expected to decrease and the local optimality condition is more likely to satisfy. Based on this intuition, we expect that $J_1$ is decreasing in $z$ and $J_2$ is increasing in $z$. A lot of the work we will done later in this section is to find the right tradeoff between the two terms.

\begin{proof}
We claim that $J_1$ is the expected number of cuts $V_1,V_2$ such that the cardinality of a cut is $zn$, $|V_1|=(1/2+\alpha)n,|V_2|=(1/2-\alpha)n$
and the number of edges within each part $V_1$ and $V_2$ is $z_1n,z_2n$, respectively. Similarly, we claim that $J_2$ is the probability that the local optimality condition
is satisfied for any given cut counted in $J_1$, where
$K$ was defined in (\ref{eq:Kmu_mubar}).   

We recall that $\Gncn$ is assumed to be generated using the configuration model. Indeed
$$
{{2cn}  \choose {{2z_1n},  {2z_2n}, {zn}, {zn}}} ((1/2+\alpha)n)^{(2z_1+z)n} ((1/2-\alpha)n)^{(2z_2+z)n}
$$
is the number of ways assigning balls to vertices, such that the vertex subset $V_1$ has $2z_1 n$ balls
(for generating $z_1 n$ edges inside it by matching them later) and another $zn$ balls (for generating $zn$ edges which cross the partition),
and, similarly, the vertex subset $V_2$ has $2z_2 n$ balls (for generating $z_2 n$ edges inside it) and another $zn$ balls (for generating $zn$ edges crossing the partition).
Now $(zn)! F(2z_1n)F(2z_2n)$ is the number of ways of creating cardinality $zn$ matchings crossing parts $V_1$ and $V_2$, cardinality $z_1 n$ matchings inside $V_1$
and cardinality $z_2 n$ matchings inside $V_2$.
$n^{2cn} F(2cn)$ is the number of ways of assigning $2cn$ balls to $n$ nodes and then randomly matching on these balls
for generating a graph with $cn$ edges. Namely, it is the total number of creating a (multi-) graph on $n$ nodes with $cn$ edges.
This establishes the claim regarding the term $J_1$.

The claim for $J_2$ follows by observing that once the size $z,z_1,z_2$ are fixed, the events that
each part $V_1$ and $V_2$ satisfies the local optimality conditions are independent and their corresponding probabilities are
$K((1/2+\alpha)n , zn, 2z_1 n)$ and  $K((1/2-\alpha)n , zn, 2z_2 n)$, respectively.
\end{proof}

We have that $z_1$ and $z_2$ satisfy $z_1+z_2=c-z$. If $2z_1  \leq z$ or $ 2z_2 \leq z$ is not true, $J_2$ is $0$ and thus does not contribute to $\mathbb{E}(X(zn))$. Then we only need to consider $2z_1  \leq z$ and $ 2z_2 \leq z$. Combining with $z_1+z_2=c-z$ and $z=c/2+x\sqrt{c}$, we have
\begin{align}
\label{z1z2}
c/2-3x\sqrt{c} &\leq  2z_1 \leq  c/2+x\sqrt{c}, \\
c/2-3x\sqrt{c} &\leq  2z_2 \leq  c/2+x\sqrt{c}. \notag
\end{align}
We use the standard approximation
\begin{align}
\label{approx_1}
\frac{1}{n} \log { n \choose {a n}}=H(a)+o(1).
\end{align}
Here and  everywhere below $H(x)$ denotes the standard entropy function
$H(x)=-x\log x-(1-x)\log(1-x)$.
We also have
\begin{align}
\label{ent_m}
\frac{1}{n} \log  {{2cn}  \choose {{2z_1n},  {2z_2n}, {zn}, {zn}}} = 2(-z_1 \log z_1 - z_2 \log z_2 - z\log z + c\log c + z \log 2 )+o(1)
\end{align}
and
\begin{align}
\label{F_mat}
\frac{1}{n} \log F(an) &=\frac{1}{n} \log \left(  \frac{(an)!}{(an/2)!2^{an/2}} \right) \nonumber \\
                      &= \frac{1}{n} \log \frac{(an)^{an} e^{-an}}{(an/2)^{an/2} e^{-an/2} 2^{an/2}} + o(1) \nonumber \\
											&= (a\log (an) -a)/2 + o(1).
\end{align}
Using (\ref{approx_1}), (\ref{ent_m}), (\ref{F_mat}) and Stirling's approximation, we have
\begin{align}
\label{J_1m}
 \frac{1}{n} \log J_1 & = H(1/2+\alpha)+2(-z_1 \log z_1 - z_2 \log z_2 - z\log z + c\log c + z \log 2) \nonumber \\
                      & \quad +(2z_1+z)\log((1/2+\alpha)n)+(2z_2+z)\log((1/2-\alpha)n) +z\log(zn)-z-2c\log n \nonumber \\
                      & \quad + (2z_1\log(2z_1 n)-2z_1+2z_2\log(2z_2 n)-2z_2-2c\log(2cn)+2c)/2 +o(1)\nonumber \\
                      & = H(1/2+\alpha) + z\log(1/4-\alpha^2) + 2z_1 \log(1/2+\alpha)+2z_2 \log(1/2-\alpha) \nonumber \\
                      & \quad -z\log z-z_1 \log 2z_1 -z_2 \log 2z_2 + c \log 2c    + o(1).
\end{align}
Using Taylor expansion $\log(1+a)=a-a^2/2+o(a^2)$, the equation above is simplified by
\begin{align}
\label{optimize_2}
  \frac{1}{n} \log J_1&=\log 2-2\alpha^2+z(-\log 4 -4 \alpha^2)
   + 2z_1(-\log2+2\alpha-2\alpha^2)
         +2z_2(-\log2-2\alpha-2\alpha^2)-z\log z \notag \\
         &-z_1 \log 2z_1 -z_2 \log 2z_2 + c \log 2c
         +o_{\alpha}(\alpha^2)c+o(1)    \nonumber \\
&=\log 2-2c\log2+4(z_1-z_2)\alpha-(2+4c)\alpha^2
-z\log z -z_1 \log 2z_1 -z_2 \log 2z_2 \notag \\
&+ c \log 2c + o_{\alpha}(\alpha^2)c+o(1).
\end{align}
Similarly, using $\log(1+a) < a$ for $|a|<1$ and $H(b) \leq \log 2$ for $b \in [0, 1]$, from (\ref{J_1m}) we also have  
\begin{align}
\label{optimize_3}
  \frac{1}{n} \log J_1 \leq& \log 2 + z (-\log 4 - 4 \alpha^2 ) + 2z_1(-\log2+2\alpha)
         +2z_2(-\log2-2\alpha)-z\log z \notag \\
         &-z_1 \log 2z_1 -z_2 \log 2z_2 + c \log 2c + o(1) \\
         =& \log 2 - 2c \log 2 + 4(z_1 - z_2)\alpha - 4z \alpha^2 -z\log z \notag -z_1 \log 2z_1 -z_2 \log 2z_2 + c \log 2c + o(1) 
\end{align}
From (\ref{z1z2}), we have $z_1-z_2=O_c(\sqrt{c})$. Recall $z = c/2 + x\sqrt{c}$.
Viewing the expression above as a quadratic form in $\alpha$, the dominating term involving $\alpha$ is
\begin{align*}
O_c(\sqrt{c})\alpha-(2c + 4 x\sqrt{c})\alpha^2  
\end{align*}
as $c$ increases. Similarly we have the the dominating term involving $\alpha$ on the right hand side of (\ref{optimize_2}) 
\begin{align}\label{eq:BoundJ1}
O_c(\sqrt{c})\alpha- 4 c \alpha^2 + o_{\alpha}(\alpha^2)c
\end{align}
as $c$ increases. Observe that the right hand sides of (\ref{optimize_2}) and (\ref{optimize_3}) share the same terms which do not depend on $\alpha$. We see that $\alpha$ which maximizes asymptotically $n^{-1}\log J_1$
should satisfy
\begin{align}\label{eq:alpha_bound}
\alpha=O_c(c^{-1/2}).
\end{align}
This result is crucial to analyze the variational problem induced by the large deviation principle underlying the evaluation of $J_2$,
the evaluation of which we now turn to.

We will evaluate $K(n,\mu_1,\mu_2)$ in (\ref{Pos_approx}) using large deviations technique, which involves the moment generating function (MGF) of two
correlated Poisson random variables. Such MGF does not unfortunately have a closed form expression.
The following lemma allows us to evaluate the MGF by that of Normal distributions for the asymptotic case $c\rightarrow \infty$.
\begin{lemma}
\label{appro_nrv}
Suppose $\lambda_1=\Theta_c(c)$, $\lambda_2=\lambda_1-\Theta_c(\sqrt{c})$, $a_c=\lambda_2/\lambda_1$, $b_c=(\lambda_2-\lambda_1)/\sqrt{\lambda_1}$.
Suppose also the limit $b=\lim_{c\rightarrow \infty}b_c$ exists.
Let $B\distr\Pois(\lambda_1)$ and $C\distr\Pois(\lambda_2)$ be two independent Poisson random variables,
and let
$X_1$ and $X_2$ be two independent standard Normal random variables. Let
\begin{align*}
(U_c,V_c)\triangleq \left(\frac{B-\lambda_1}{\sqrt{\lambda_1}},\frac{C-\lambda_2}{\sqrt{\lambda_2}}\right).
\end{align*}
For every fixed $\theta_1,\theta_2$
\begin{align}
\label{appro_mgf}
\lim_{c\rightarrow\infty}\mathbb{E}[\exp(U_c\theta_1+V_c\theta_2) \mid U_c \geq \sqrt{a_c}V_c+b_c ] =
\mathbb{E}[\exp(X_1\theta_1+X_2 \theta_2) \mid X_1 \geq X_2+b].
\end{align}
\end{lemma}
\begin{proof}
We have that $(U_c,V_c)$ converges in distribution to $(X_1,X_2)$ as $c\rightarrow\infty$. The result then follows by observing
uniform integrability of $(U_c,V_c)$ as $c\rightarrow\infty$, which implies convergence in expectation.
\end{proof}

Applying large deviations estimation of Theorem \ref{LCLTLDP} and Lemma \ref{appro_nrv},
we compute the conditional probability underlying $K(n,\mu_1,\mu_2)$ as follows.
\begin{lemma}\label{LPD_lemma}
Suppose $\mu_j=\mu_j(n,c), j=1,2$  are positive integer sequences such that $\lambda_j=\lim_n\mu_j/n, j=1,2$ exist for every $c$, take rational values and satisfy
$\lambda_1=\Theta_c(c), \lambda_2=\lambda_1-\Theta_c(\sqrt{c})$.
Suppose further that the limit $b=\lim_{c\rightarrow \infty} \frac{\lambda_2-\lambda_1}{\sqrt{\lambda_1}}$ exists and satisfies $b<0$.
Let $B_i,C_i,1\leq i \leq n$ be i.i.d. Poisson random variables with mean $\E[B_i]=\lambda_1,\E[C_i]=\lambda_2$. Then
\begin{align}
\label{P_BiBbari}
& \lim_{n\rightarrow \infty} \frac{1}{n} \log \mathbb{P}\left[\sum_{i=1}^{n} B_i=\mu_1, \sum_{i=1}^{n} C_i=\mu_2  \Bigm\vert B_i \geq C_i, 1\leq i\leq n \right]  \\
\label{sup_charac_1}
& =  -\log(2 P_1)-\sup_{\theta \in \mathbb{R}}(-\theta^2-\log(1+\erf(\theta-b/2)))+o_c(1)
\end{align}
where $P_1=\mathbb{P}[X_1\geq X_2+b]$, $X_1$ and $X_2$ are two independent standard normal random variables.
Furthermore, the equation (\ref{transce2}) has a unique solution for any $x>0$ and (\ref{sup_charac_1}) can be rewritten by
\begin{align}
\label{pro_bloc}
  -\log(2P_1) +(\theta^*)^2+ \log(1+\erf(\theta^{*}-b/2)) +o_c(1),
\end{align}
where $\theta^*$ is the unique solution to (\ref{transce2}) for $x=-b/4$.
\end{lemma}

\begin{proof}[Proof of Lemma~\ref{LPD_lemma}]
Let $U_i=\frac{B_i-\lambda_1}{\sqrt{\lambda_1}}$ and $V_i=\frac{C_i-\lambda_2}{\sqrt{\lambda_2}}$. In order to use Lemma \ref{appro_nrv} to approximate the MGF involved in the computation of (\ref{P_BiBbari}), we rewrite the probability term in (\ref{P_BiBbari}) by
\begin{align}
\label{new_rv2}
\mathbb{P}\left[\sum_{i=1}^n U_i=0, \sum_{i=1}^n V_i=0 \middle|  U_i \geq \sqrt{\frac{\lambda_2}{\lambda_1}} V_i+ \frac{\lambda_2-\lambda_1}{\sqrt{\lambda_1}}, 1\leq i\leq n  \right].
\end{align}
Conditional on $U_i \geq \sqrt{\frac{\lambda_2}{\lambda_1}} V_i+ \frac{\lambda_2-\lambda_1}{\sqrt{\lambda_1}}$, the joint distribution of $(U_i, V_i)$, $1\leq i \leq n$,
defines a new sequence of i.i.d. random variables $(\tilde U_i,\tilde V_i) \in \mathbb{R}^2$, where $(\tilde U_i,\tilde V_i)$ have the distribution
of $(U_i,V_i)$ conditional on $U_i \geq \sqrt{\frac{\lambda_2}{\lambda_1}} V_i+ \frac{\lambda_2-\lambda_1}{\sqrt{\lambda_1}}$.
Since $\lambda_i$ take rational values, $(0,0)$ belongs to the lattice $\mathcal{L}_n$ supporting $(U_i,V_i)$ for all sufficiently large $n$.
Applying Theorem \ref{LCLTLDP} we have
\begin{align*}
\lim_{n\rightarrow \infty} &\frac{1}{n} \log \mathbb{P}\left[\sum_{i=1}^n U_i=0, \sum_{i=1}^n V_i=0 \middle|  U_i \geq \sqrt{\frac{\lambda_2}{\lambda_1}} V_i+ \frac{\lambda_2-\lambda_1}{\sqrt{\lambda_1}}, 1\leq i\leq n  \right] \\
&=\lim_{n\rightarrow \infty} \frac{1}{n} \log \mathbb{P}\left[\sum_{i=1}^n \tilde U_i=0, \sum_{i=1}^n \tilde V_i=0\right] \\
&= - I(0,0),
\end{align*}
where $I(x_1,x_2)=\sup_{(\theta_1,\theta_2) \in \mathbb{R}^2} (\theta_1 x_1+\theta_2 x_2-\log(M(\theta_1,\theta_2))$ is the rate function, and $M(\theta_1,\theta_2)$ is the MGF of the newly defined random variables $(\tilde U_i,\tilde V_i)$. For $c$ sufficiently large, Lemma \ref{appro_nrv} yields
\begin{align}
& M(\theta_1,\theta_2) =\mathbb{E}[\exp(\theta_1 \tilde U_1+\theta_2 \tilde V_1)]
=\mathbb{E}\left[\exp(\theta_1 U_1+\theta_2 U_2) \mid U_1 \geq \sqrt{\frac{\lambda_2}{\lambda_1}} U_2+ \frac{\lambda_2-\lambda_1}{\sqrt{\lambda_1}}\right] \nonumber \\
\label{MGF2_0}
                     &=    \frac{1}{P_1} \iint \limits_{t_1 \geq t_2 + b} \,  \frac{1}{2\pi} \exp(\theta_1 t_1 + \theta_2 t_2) \exp\left( -\frac{t_1^2+t_2^2}{2} \right) d t_1\,d t_2+o_c(1)   \\
										\label{MGF2_1}
										&=\frac{\exp(\frac{\theta_1^2+\theta_2^2}{2})}{P_1} \iint \limits_{t_1 \geq t_2+b} \,  \frac{1}{2\pi} \exp \left(-\frac{(t_1-\theta_1)^2}{2}-\frac{(t_2-\theta_2)^2}{2}  \right) d t_1\,d t_2+o_c(1)  \\
										\label{MGF2_2}
&=\frac{\exp(\frac{\theta_1^2+\theta_2^2}{2})}{P_1} \iint \limits_{\mathcal{D}_1} \, \frac{1}{2\pi} \exp(-\frac{\bar{t}_1^2+\bar{t}_2^2}{2}) d \bar{t}_1\,d \bar{t}_2+o_c(1)   \\
                    \label{MGF2_2_1}
&=\frac{\exp(\frac{\theta_1^2+\theta_2^2}{2})}{P_1} \int_{\frac{\theta_2-\theta_1+b}{\sqrt{2}}}^{\infty} \, \frac{1}{\sqrt{2\pi}} \exp \left(-\frac{t^2}{2}\right) d t +o_c(1)  \\
                    \label{MGF2_3}
&=\frac{\exp(\frac{\theta_1^2+\theta_2^2}{2})}{P_1} \frac{1+\erf \left( \frac{\theta_1-\theta_2-b}{2} \right)}{2}+o_c(1),
\end{align}
where from (\ref{MGF2_1}) to (\ref{MGF2_2}) we have used the change of variables $\bar{t}_1=t_1-\theta_1$ and $\bar{t}_2=t_2-\theta_2$ to simplify the integral, and $\mathcal{D}_1$ is
$$
\mathcal{D}_1=\{(\bar{t}_1,\bar{t}_2): \bar{t}_1 \geq \bar{t}_2 + \theta_2-\theta_1+b \}.
$$
Then we have
\begin{align}
\label{logMGF}
\log M(\theta_1,\theta_2)=\frac{\theta_1^2+\theta_2^2}{2}-\log (2P_1) + \log \left(1+\erf \left( \frac{\theta_1-\theta_2-b}{2} \right) \right)+o_c(1),
\end{align}
and the large deviations rate function valued at $(0,0)$ is
$$
I(0, 0)=\sup_{(\theta_1,\theta_2) \in \mathbb{R}^2} \{ - \log M(\theta_1,\theta_2) \}
$$
From (\ref{logMGF}), we have that $\log M(\theta_1,\theta_2)<\infty$ for any $(\theta_1,\theta_2) \in \R^2$.
This implies (see Exercise 2.2.24 and its hint in~\cite{demzei98}) that $\log M(\theta_1,\theta_2)$ is a strictly convex function and its minimum is achieved at a unique
point $\theta^*=(\theta_1^*,\theta_2^*)$ at which the gradient of $\log M(\theta_1,\theta_2)$ vanishes. Namely, we have at $\theta^*$
\begin{align}
\label{cond1}
 \frac{\partial \log M(\theta_1,\theta_2)}{\partial \theta_1}=0 &\Rightarrow  \theta_1 + \frac{1}{1+\erf((\theta_1-\theta_2-b)/2)}
 \frac{\partial (\erf((\theta_1-\theta_2-b)/2))}{\partial \theta_1} =0 \nonumber \\
& \Rightarrow  \theta_1 + \sqrt{\frac{1}{\pi}} \frac{e^{-(\theta_1-\theta_2-b)^2/4}}{1+\erf((\theta_1-\theta_2-b)/2)} =0.
\end{align}
Likewise, we have
\begin{align}
\label{cond2}
& \frac{\partial \log M(\theta_1,\theta_2)}{\partial \theta_2}=0 \Rightarrow \theta_2  - \sqrt{\frac{1}{\pi}}
\frac{e^{-(\theta_1-\theta_2-b)^2/4}}{1+\erf((\theta_1-\theta_2-b)/2)}=0.
\end{align}
From (\ref{cond1}) and (\ref{cond2}), we observe that $\theta_2=-\theta_1$. Then (\ref{cond1}) and (\ref{cond2}) becomes the same equation, which
can be rewritten by (\ref{transce2}) with $x=-b/4$. Hence (\ref{transce2}) has a unique solution for any $x>0$. Let the unique solution for $x=-b/4$ be $\theta^*$.
The rate funtion $I(\cdot,\cdot)$ at $(0,0)$ is
\begin{align}
I(0,0)&=\sup_{\theta \in \mathbb{R}}(-\theta^2-\log(1+\erf(\theta-b/2)))+\log(2P_1) \\
&=-(\theta^*)^2-\log(1+\erf(\theta^*-b/2))+\log(2P_1).
\end{align}
\end{proof}
We now introduce the following form of  reverse  H\"{o}lder's inequality~\cite{gardner2002brunn}.
\begin{lemma}[Pr\'{e}kopa--Leindler inequality]
Let $\lambda \in (0,1)$ and let $f$, $g$, $h: \mathbb{R}^n \rightarrow [0, +\infty)$ be non-negative real-valued measurable functions defined on $\mathbb{R}^n$. Suppose that these functions satisfy
$$
h((1-\lambda)x+\lambda y) \geq f(x)^{1-\lambda}g(y)^{\lambda}
$$
for all $x$ and $y$ in $\mathbb{R}^n$. Then
$$
\|h\|_1=\int_{\mathbb{R}^n} h(x)dx \geq \left( \int_{\mathbb{R}^n} f(x) dx \right)^{1-\lambda} \left( \int_{\mathbb{R}^n} g(x) dx \right)^{\lambda}= \| f \|_1^{1-\lambda} \|g\|_1^{\lambda}
$$
\end{lemma}
\begin{lemma}
\label{maximize_K2d}
Let
\begin{align}
L(a)=-\sup_{\theta \in \mathbb{R}} \left( -\theta^2 - \log \left( 1+\erf(\theta-a/\sqrt{2})\right)  \right).
\end{align}
Then $L(a)$ is a concave function for $a\in \mathbb{R}$.
\end{lemma}
\begin{proof}
The function $f(t)=1/\sqrt{2\pi}\exp(-t^2/2)$ is log-concave,
which implies that for all $\theta \in \mathbb{R},a_1, a_2 \in \mathbb{R},t_1, t_2 \in \mathbb{R}$ and $\lambda \in (0,1)$, we have
\begin{align}
\label{log_concave_f}
& f(\lambda(t_1+a_1-\sqrt{2}\theta)+(1-\lambda)(t_2+a_2-\sqrt{2}\theta)) \nonumber \\
& \geq f^{\lambda}(t_1+a_1-\sqrt{2}\theta)f^{1-\lambda}(t_2+a_2-\sqrt{2}\theta).
\end{align}
Let
\begin{align}
\label{hg1_g2}
& h(t) = f(t+(\lambda a_1 + (1-\lambda)a_2)-\sqrt{2}\theta)  \mathbf{1}_{ t\geq 0 },  \nonumber \\
& g_1(t) = f(t+a_1-\sqrt{2}\theta) \mathbf{1}_{t \geq 0}, \nonumber \\
& g_2(t) = f(t+a_2-\sqrt{2}\theta)\mathbf{1}_{ t\geq 0 },
\end{align}
where $\mathbf{1}_{ t\geq 0 }$ is the indicator function. $h(t), g_1(t)$ and $g_2(t)$ are non-negative functions, which by (\ref{log_concave_f}) satisfy
$$
h(\lambda t_1 + (1-\lambda) t_2) \geq g_1^{\lambda}(t_1)g_2^{1-\lambda}(t_2),
$$
for any $t_1, t_2 \in \mathbb{R}$. Then Pr\'{e}kopa--Leindler inequality gives that
\begin{align}
& \int_{\mathbb{R}} h(t) dt \geq  \left(\int_{\mathbb{R}} g_1(t) dt \right)^{\lambda} \left(\int_{\mathbb{R}} g_2(t) dt \right)^{1-\lambda}. \nonumber
\end{align}
Namely
\begin{align}
 &  \int_0^{\infty} f(t+(\lambda a_1 + (1-\lambda)a_2)-\sqrt{2}\theta) dt  \nonumber \\
& \geq \left(\int_0^{\infty} f(t+a_1-\sqrt{2}\theta) dt \right)^{\lambda} \left(\int_0^{\infty} f(t_1+a_2-\sqrt{2}\theta) dt \right)^{1-\lambda} \nonumber
\end{align}
Taking  $-\log$ of both sides we have
\begin{align}
& -\log \left(\int_0^{\infty} f(t+(\lambda a_1 + (1-\lambda) a_2)-\sqrt{2}\theta) dt \right) \nonumber  \\
& \leq \lambda \left(-\log \left(\int_0^{\infty} f(t+a_1-\sqrt{2}\theta) dt \right) \right) + (1-\lambda) \left(-\log \left( \int_0^{\infty} f(t+a_2-\sqrt{2}\theta) dt \right) \right)
\end{align}
which yields that for any $\theta \in \mathbb{R}$,
$$
-\log \left(\int_0^{\infty} f(t+a-\sqrt{2}\theta) dt\right)=-\log (1+\erf(\theta-a/\sqrt{2}))+\log 2
$$
is a convex function in $a$. Since the pointwise supremum of convex functions is convex, we obtain that
$$
\sup_{\theta\in \mathbb{R}} \left(-\theta^2-\log (1+\erf(\theta-a/\sqrt{2})) \right)
$$
is also a convex function in $a$, and then the concavity of $L(a)$ follows.
\end{proof}

With Lemmas \ref{lemma_twoPo}, \ref{LPD_lemma} and \ref{maximize_K2d}, we are now ready to consider the problem of maximizing $J_2$ over $z_1$ and $z_2$
and complete the proof of Proposition~\ref{prop:comp_EbarX}.

\begin{proof}[Proof of Proposition~\ref{prop:comp_EbarX}]
Introduce $\tau$ by
$$
2z_1=c/2-(x+\tau)\sqrt{c} ,\; 2z_2=c/2-(x-\tau)\sqrt{c}\;
$$
From (\ref{z1z2}) we can restrict $\tau$ to be in the range $\tau \in(-2x,2x)$.
Then
$$
\frac{\frac{2z_1}{(1/2+\alpha)}-\frac{z}{(1/2+\alpha)}}{\sqrt{\frac{z }{(1/2+\alpha)}}}= -{4x -2\tau \over \sqrt{1/2+\alpha}}.
$$
First applying (\ref{Pos_approx}) in Lemma \ref{lemma_twoPo}, and then (\ref{sup_charac_1}) in Lemma \ref{LPD_lemma} yields
\begin{align*}
 & \lim_{n\rightarrow \infty} \frac{1}{n} \log K((1/2+\alpha)n , zn, 2z_1 n) \nonumber \\
=&  (\frac{1}{2}+\alpha) \left(-\log 2-\sup_{\theta \in \mathbb{R}}(-\theta^2-\log(1+\erf(\theta+(1/2+\alpha)^{-{1\over 2}}(2x+\tau))) )\right)+o_c(1).
\end{align*}
where we have canceled out the same term $\log P_1$ from (\ref{Pos_approx}) and (\ref{sup_charac_1}). This expression is bounded from above
uniformly in $\alpha\in [0,1/2]$
as $c$ increases because of the $o_c(1)$ term. Likewise we have
\begin{align*}
& \lim_{n\rightarrow \infty} \frac{1}{n} \log K((1/2-\alpha)n , zn, 2z_2 n)  \nonumber \\
=&  (\frac{1}{2}-\alpha)  \left(-\log 2-\sup_{\theta \in \mathbb{R}}(-\theta^2-\log(1+\erf(\theta+(1/2-\alpha)^{-{1\over 2}}(2x-\tau)))) \right)+o_c(1),
\end{align*}
which is also uniformly bounded from above in $\alpha\in [0,1/2]$  as $c$ increases. Recalling (\ref{eq:BoundJ1}) we see that we may assume
$\alpha=O_c(c^{-1/2})$ which gives
\begin{align}
\label{K_z_z1}
 & \lim_{n\rightarrow \infty} \frac{1}{n} \log K((1/2+\alpha)n , zn, 2z_1 n) \nonumber \\
=&  -\frac{1}{2} \log 2-\frac{1}{2} \sup_{\theta \in \mathbb{R}}(-\theta^2-\log(1+\erf(\theta+2x+\tau))) +o_c(1)
\end{align}
where we have canceled out the same term $\log P_1$ from (\ref{Pos_approx}) and (\ref{sup_charac_1}). Likewise we have
\begin{align}
\label{K_z_z2}
& \lim_{n\rightarrow \infty} \frac{1}{n} \log K((1/2-\alpha)n , zn, 2z_2 n)  \nonumber \\
=&  -\frac{1}{2} \log 2-\frac{1}{2} \sup_{\theta \in \mathbb{R}}(-\theta^2-\log(1+\erf(\theta+2x-\tau))) )+o_c(1)
\end{align}
Using Lemma \ref{maximize_K2d}, (\ref{K_z_z1}) and (\ref{K_z_z2}) is combined by
\begin{align}
\label{Kz1_z2_ub}
& \lim_{n \rightarrow \infty}\frac{1}{n} \log J_2= \lim_{n \rightarrow \infty}\frac{1}{n} \log K((1/2+\alpha)n , zn, 2z_1 n)K((1/2-\alpha)n , zn, 2z_2 n)  \nonumber \\
& = -\log 2 +\frac{1}{2} L(-2\sqrt{2}x-\sqrt{2}\tau) + \frac{1}{2} L(-2\sqrt{2}x+\sqrt{2}\tau) + o_c(1) \nonumber \\
& \leq  -\log 2 + L(-2\sqrt{2}x) +o_c(1) = -\log 2 - \sup_{\theta \in \mathbb{R}} (-\theta^2 - \log(1+\erf(\theta+2x))) + o_c(1)
\end{align}
where the equality holds when $\tau=0$, i.e. $z_1=z_2=c/4-x\sqrt{c}/2$, which corresponds to the cut under which the number of edges within
each part is the same.  The supremum in (\ref{Kz1_z2_ub}) is attained by the solution $\theta$ to (\ref{transce2}).

Now we go back to optimizing $\frac{1}{n} \log J_1$ in (\ref{optimize_2}), while relying on the bound (\ref{eq:alpha_bound}).
Consider the right-hand side of (\ref{optimize_2}). Consider this expression without the $o_{\alpha}(\alpha^2)$ and $o(1)$ terms and denote it
by $V(\alpha,z_1)$ after substitution $z_2=c-z-z_1$. Since we have already established that $\alpha=O_c(c^{-1/2})$, this is justified in the later steps.
We have
\begin{align}
\label{V_az1}
V(\alpha,z_1)=&\log 2 - 2c\log2 + 4(2z_1-c+z)\alpha-(2+4c)\alpha^2-z \log z - z_1 \log2z_1 \nonumber \\
              & - (c-z-z_1)\log 2(c-z-z_1) + c\log 2c,
\end{align}
where $z_1$ is subject to (\ref{z1z2}). $\alpha=\frac{2z_1-(c-z)}{1+2c}$ maximizes $V(\alpha,z_1)$. Substituting it to (\ref{V_az1}) yields
\begin{align}
V\left(\frac{2z_1-(c-z)}{1+2c},z_1 \right)= &\log 2 - 2c\log2 -z \log z + c\log 2c + \frac{2(2z_1-c+z)^2}{1+2c}- \nonumber \\
                                &-z_1 \log2 z_1 - (c-z-z_1) \log2(c-z-z_1).
\end{align}
Its first derivative w.r.t. $z_1$ is
\begin{align}
\label{f_deri}
\frac{8(2z_1-c+z)}{1+2c}-\log(2z_1)+\log2(c-z-z_1),
\end{align}
and its second derivative w.r.t. $z_1$ is
$$
\frac{16}{1+2 c}-\frac{1}{c-z-z_1}-\frac{1}{z_1}.
$$
It is easy to see that the expression above is maximized at $(c-z)/2$, which yields
$$
\frac{16}{1+2 c}-\frac{2}{c-z}-\frac{2}{c-z}= \frac{8}{\frac{1}{2}+ c}-\frac{8}{c-2x\sqrt{c}}<0.
$$
Hence, $V(\frac{2z_1-(c-z)}{1+2c},z_1)$ is concave in $z_1$. Setting its first derivative in (\ref{f_deri}) to zero, namely,
$$
\frac{8(2z_1-c+z)}{1+2c}-\log(2z_1)+\log2(c-z-z_1)=0
$$
gives that $V(\frac{2z_1-(c-z)}{1+2c},z_1)$ is maximized at $z_1=(c-z)/2=c/4-x\sqrt{c}/2$, which is the same as the condition $\tau=0$ for maximizing $\frac{1}{n}\log J_2$. In other words, $J_1$ and $J_2$ attain the maximum under the same conditions $\alpha=0$ and $z_1=z_2=(c-z)/2$. Substituting $z=c/2+x\sqrt{c}$ and using the asymptotic expansion
$$
\log(c/2+x\sqrt{c})=\log(c/2)+2x/\sqrt{c}-2x^2/c+o_c(1/c)
$$
simplify the maximum of $V(\alpha,z_1)$ as
\begin{align}
\label{V_opt}
 V(0,(c-z)/2)
&= \log 2 - 2c\log2 -z \log z + c\log 2c - (c-z) \log (c-z)+o_c(1)+o(1) \nonumber \\
&=\log 2 - c \log 2 +c\log c -(c/2+x\sqrt{c})\log(c/2+x\sqrt{c}) \nonumber \\
 & \quad - (c/2-x\sqrt{c})\log(c/2-x\sqrt{c}) +o_c(1)+o(1) \nonumber \\
&= \log2-2x^2+o_c(1)+o(1).
\end{align}
Combining the results in (\ref{Kz1_z2_ub}) and (\ref{V_opt}), we have that the exponent of $\mathbb{E}[X(zn)]$ in (\ref{halfalpha}) is attained at $\alpha=0$ and $z_1=(c-z)/2$, i.e.
$$
\lim_{n\rightarrow \infty} \frac{1}{n} \log \mathbb{E}[X(zn)]=-2x^2-\sup_{\theta\in \mathbb{R}}(-\theta^2-\log(1+\erf(2x+\theta)))+o_c(1).
$$
Then (\ref{EtildeEX}) follows from solving $\theta$ from (\ref{transce2}) for a given $x$. This completes the proof of Proposition~\ref{prop:comp_EbarX}.
\end{proof}

Finally we prove Lemma~\ref{lemma:UniqueSolution}.
\begin{proof}[Proof of Lemma~\ref{lemma:UniqueSolution}]
Lemma~\ref{LPD_lemma} gives that for every $x$ in the region (\ref{eq:xupperrange}),
there exists a unique solution $\theta(x)$ to the equation (\ref{transce2}). We now use it to establish the uniqueness of the solution to the equation
system (\ref{transce1}), (\ref{transce2}).

It can be verified using elementary methods
 that $w(x)$ defined by (\ref{eq:wx}) is a differentiable function on $\R$, and therefore is continuous.
We verify numerically that $w(0.3761)=0.19721..>0>w(0.5887) =-0.05595..~$. Then the existence of the solution follows by the continuity of $w(x)$.

We now establish the uniqueness of the solution. We have
\begin{align*}
{\dot w}(x)&=-4x+2\theta(x) \frac{d \theta(x)}{d x}  +\frac{e^{-(2x+\theta(x))^2}}{\sqrt{\pi}(1+\erf(2x+\theta(x)))} \left(4+2 \frac{d \theta(x)}{d x} \right) \nonumber  \\
         &=-4x+2\theta(x) {d\theta(x)\over dx} -\theta(x) \left(4+2 {d\theta(x)\over dx} \right) \nonumber  \\
				 &=-4(x+\theta(x)),  \nonumber
\end{align*}
where we have used (\ref{transce2}) to simplify the first step above. Next, we will show that $x+\theta(x) > 0$. This implies that $w$ is a strictly decreasing
function in the relevant region and therefore the solution is unique as claimed.
Letting $y(x)=x+\theta(x)$,  (\ref{transce2}) can be rewritten by
$$
\frac{1}{\sqrt{\pi}} \frac{e^{-(x+y(x))^2}}{1+\text{erf}(x+y(x))}=x-y(x)
$$
Since $\theta(x)$ is unique for a fixed $x \in [0.3761, 0.5887]$, $y(x)$ is also the unique solution to the equation above. For a fixed $x \in [0.3761, 0.5887]$, let
$$
g(y)=x-y-\frac{1}{\sqrt{\pi}} \frac{e^{-(x+y)^2}}{1+\text{erf}(x+y)}
$$
We check numerically that $g(0)=x-\frac{1}{\sqrt{\pi}} \frac{e^{-x^2}}{1+\text{erf}(x)}>0$ and $g(x)=-\frac{1}{\sqrt{\pi}} \frac{e^{-4x^2}}{1+\text{erf}(2x)}<0$ for any $x$
in the region (\ref{eq:xupperrange}). Therefore the unique solution to $g(y)=0$ belongs to the region $(0,x)$ and therefore $y(x)>0$ as claimed.
\end{proof}

\section{Lower bound. The second moment method}\label{section:LowerBound}
In this section, we use the second moment method coupled with the local optimality property of optimal cuts to obtain a lower bound on the optimal cut size.
Specifically, let again $X(zn)$ denote the number of cuts with value $zn$ satisfying the local optimality constraints. For even $n$, we restrict this set to consist of balanced
cuts only, $\alpha=0$, with left and right node sets having the same cardinality, still denoting this set by $X(zn)$ for convenience. For odd $n$, we instead restrict this set to the cuts with an imbalance of one node between the sides of the cut. Then $\alpha = O(1/n)$. We will ignore this case as $\alpha = O(1/n)$ only contributes an extra term $O(1/n)$ to $\frac{1}{n} \mathbb{E}[X^2(zn)]$. The main bulk of this
section will be devoted to showing the following result, which is an analogue of Proposition~\ref{prop:comp_EbarX} for the second moment computation.
\begin{proposition}
\begin{align}\label{eq:SecondMoment}
\E[X^2(zn)]\le \E^2[X(zn)]\exp(o_c(1)n),
\end{align}
when $z=c/2+x\sqrt{c}$ and $x<x_l$ with $x_l$ identified in Theorem~\ref{main_theorem}.
\end{proposition}

Before proving it let us show how it implies
the lower bound part of Theorem~\ref{main_theorem}.
Using inequality
\begin{align*}
\pr(X(zn)\ge 1)\ge {\E^2[X(zn)]\over \E[X^2(zn)]},
\end{align*}
which is a well-known and easy implication of the Cauchy-Schwartz inequality, we obtain $\pr(X(zn)\ge 1)\ge \exp(-o_c(1)n)$.
From this we obtain
\begin{align}
\label{MC_inequality1}
\mathbb{P}[MC_{c,n} \geq zn] \geq \exp(-no_c(1)).
\end{align}
Applying the Hoeffding-Azuma inequality for Max-Cut gives
\begin{align}
\label{Azuma_MC}
\mathbb{P}[\lvert MC_{c,n}-\mathbb{E}[MC_{c,n}]  \rvert \geq t] \leq 2 \exp\left( -\frac{t^2}{2cn} \right), \; \text{for} \;  t>0.
\end{align}
Set $t=\frac{\epsilon}{2}\sqrt{c}n$ for any given $\epsilon>0$, and then (\ref{Azuma_MC}) gives
\begin{align}
\label{HA_inequality2}
\mathbb{P}[\lvert MC_{c,n}-\mathbb{E}[MC_{c,n}]  \rvert \geq \frac{\epsilon}{2}\sqrt{c}n] \leq 2 \exp \left(-\frac{\epsilon^2}{8}n \right).
\end{align}
Observe that for any given $\epsilon>0$, $c$ can be chosen sufficiently large such that $\exp(-no_c(1))$ on the right-hand side  of (\ref{MC_inequality1}) is larger than $2 \exp \left(-\frac{\epsilon^2}{8}n \right)$. Then
$$
\mathbb{E}[MC_{c,n}] + \frac{\epsilon}{2} \sqrt{c} n \geq z n.
$$
Applying (\ref{HA_inequality2}) again yields that w.h.p.
$$
MC_{c,n} \geq \mathbb{E}[MC_{c,n}]-\frac{\epsilon}{2}\sqrt{c}n \geq  zn -\epsilon \sqrt{c} n
$$
Since the inequality holds for any $\epsilon>0$, it yields that w.h.p.
$$
\liminf_{c\rightarrow \infty} \lim_{n\rightarrow \infty} \frac{ \frac{1}{n}MC_{c,n}-\frac{c}{2}}{\sqrt{c}}
=\lim_{c\rightarrow \infty} \frac{ \mathcal{MC}(c)-\frac{c}{2}}{\sqrt{c}}
\geq x,
$$
for every $x<x_l$,
and we obtain the result.

We now establish (\ref{eq:SecondMoment})
As in Section~\ref{section:UpperBound}, we use the following occupancy problems to analyze the local optimality condition.
Consider four independent experiments. In the experiment $j=1,\dots,4$, $\mu_j$ balls are thrown independently and u.a.r. into $n$ bins.
Denote the number of balls in bin $i$, $i\in[n]$, as $E_i^{(j)}$. Let
$$
K(n,\mu_1,\mu_2,\mu_3,\mu_4)=\mathbb{P}[E_i^{(2)}-E_i^{(4)} \geq \lvert E_i^{(1)}-E_i^{(3)} \rvert, 1 \leq i \leq n],
$$
which will be used to compute the probability of the event that all the vertices in a specified vertex subset satisfy the local optimality conditions for a pair of cuts. Similar to Lemma \ref{lemma_twoPo}, we have the following Lemma,
\begin{lemma}
\label{lemmaK4d}
For each $j$, $j=1,\dots,4$, Let $(B_i^{(j)})_{i\in [n]}$ be a family of i.i.d. Poisson variables with the same mean $\lambda_j=\mu_j/n>0$, then we have
\begin{align}\label{Km_4d}
K(n,\mu_1,\mu_2,\mu_3,\mu_4)
= \prod_{j=1}^4 \Theta_{\mu_j}(\sqrt{\mu_j})\mathbb{P}\left[\sum_{i=1}^{n} B_i^{(j)}=\mu_j, 1 \leq j \leq 4   \Biggl| \mathcal{B}_i, 1 \leq i \leq n \right]
                                  (\mathbb{P}[ \mathcal{B}_1 ])^{n},
\end{align}
where $\mathcal{B}_i$ denotes the event $B_i^{(2)}-B_i^{(4)} \geq \lvert B_i^{(1)}-B_i^{(3)} \rvert$.
\end{lemma}

\begin{figure}[!htb]
\vspace{0pc}
\centering
\includegraphics[width=2 in]{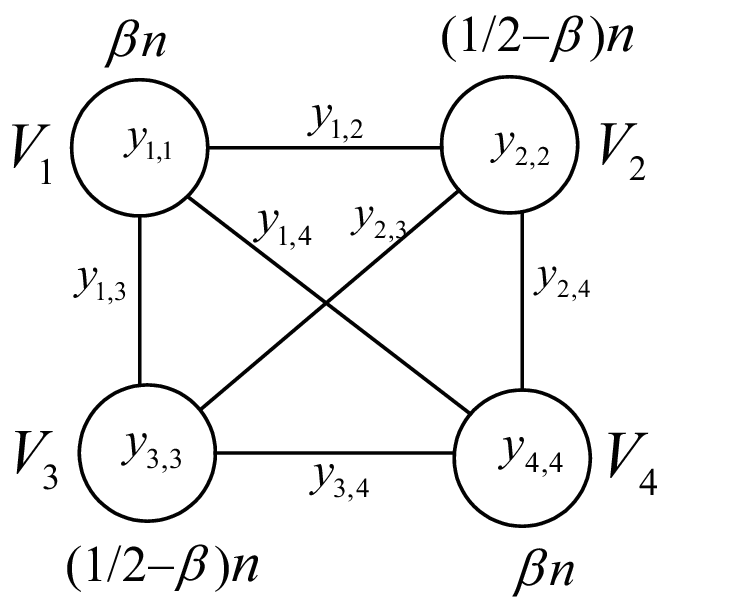}
\caption{Illustration of computing $E[X^2(zn)]$ where $y_{j,k}=nz_{j,k}$, $1 \leq j \leq k\leq 4$.}
\label{MCSM}
\end{figure}

Consider a partition of the vertex set $[n]$ into four subsets $V_j, 1\le j\le 4$ with cardinalities
\begin{align}\label{V1V4}
|V_1|=|V_4|=\beta n, ~~|V_2|=|V_3|=(1/2-\beta)n.
\end{align}
Each such partition defines two cuts. The first cut is $(D_1,D_2)$, where $D_1=V_1\cup V_3, D_2=V_2\cup V_4$. The second cut is  $(D_3,D_4)$, where $D_3=V_1\cup V_2, D_4=V_3\cup V_4$.
Let $nz_{j,k}$, $j,k \in \{1,2,3,4\}$ and $j<k$, be the number of edges crossing $V_j$ and $V_k$ in the random graph
$\Gncn$. For the case of $j=k$, $nz_{j,j}$ is the number of edges with both ends in the same vertex subset $V_j$.

Let $\beta_{j,k}$ be the number of all the possible edges crossing $V_j$ and $V_k$ normalized by $n^2$.
Namely $\beta_{j,k}=|V_j||V_k|/n^2$ when $j\ne k$ and $\beta_{j,j}=|V_j|(|V_j|+1)/(2n^2)$. If $(D_1,D_2)$ defines a cut of size $zn$, then
it must be the case that $z_{1,2}+z_{1,4}+z_{2,3}+z_{3,4}=z$. Similarly, if $(D_3,D_4)$ defines a cut of size $zn$, then
it must be the case that $z_{1,3}+z_{1,4}+z_{2,3}+z_{2,4}=z$.  We use this observation to state and proof the following result.

\begin{lemma}\label{lemma:Optimization2ndMoment}
\begin{align}
\label{EXm}
\mathbb{E}[X^2(zn)] = &   \sum_{\beta, nz_{j,k}} I_1 I_2,
\end{align}
where the sum is over all $\beta,z_{j,k}$ such that $\beta\in [0,1/4]$, $(1/2-\beta)n, \beta n, z_{j,k}n$ are integers and
\begin{align}
I_1=& \frac{{n \choose {\beta n} , {(1/2-\beta)n}, {(1/2-\beta)n}, {\beta n}}  }{n^{2cn} F(2cn)} \times I_{1,1} I_{1,2} I_{1,3}  \nonumber \\
I_{1,1} &= {(2cn)! \over \prod_{1\le i\le 4}(2nz_{i,i})!\prod_{1\le i<j\le 4}((nz_{i,j})!)^2} \nonumber \\
I_{1,2}&=(\beta n)^{(2z_{1,1}+2z_{4,4}+2z_{1,4}+z_{1,2}+z_{1,3}+z_{2,4}+z_{3,4})n}
 ((1/2-\beta)n)^{(2z_{2,2}+2z_{3,3}+2z_{2,3}+z_{1,2}+z_{1,3}+z_{2,4}+z_{3,4})n},   \\
I_{1,3} &= \prod_{1\leq i <j \leq 4} (nz_{i,j})! \prod_{i=1}^4 F(2nz_{i,i}),
\end{align}
\begin{align}
I_2 =& K(\beta n, nz_{1,2},nz_{1,4},nz_{1,3},2nz_{1,1})
     K((1/2-\beta) n, nz_{1,2}, nz_{2,3},nz_{2,4}, 2nz_{2,2})\times  \nonumber \\
    &  \times K((1/2-\beta) n, nz_{3,4},nz_{2,3},nz_{1,3},2nz_{3,3})
		K(\beta n, nz_{3,4},nz_{1,4},nz_{2,4},2nz_{4,4}),
\end{align}
and $z_{j,k}\ge 0$, $0\le j\le k \le 4$ satisfy
\begin{align}
\label{y1_y4}
& z_{1,2}+z_{1,4}+z_{2,3}+z_{3,4}=z, \nonumber \\
& z_{1,3}+z_{1,4}+z_{2,3}+z_{2,4}=z, \nonumber \\
& \sum_{1\le j\le k\le 4} z_{j,k}=c.
\end{align}

\ignore{
\begin{align}
\label{Y_jk}
\beta_{j,k} =
\begin{cases}
\lvert V_j \rvert \lvert V_k \rvert, & \text{if } j <k \\
 {\lvert V_j \rvert} \choose 2  & \text{if } j=k
\end{cases}
\end{align}
where $\lvert V_j \rvert$, $1 \leq j\leq 4$ are given by (\ref{V1V4}).
}

\end{lemma}
\begin{proof}
The term ${n \choose {\beta n} , {(1/2-\beta)n}, {(1/2-\beta)n}, {\beta n}}$ in $I_1$ is the number of ways of selecting sets $V_j$ satisfying (\ref{V1V4}).
As shown in Figure \ref{MCSM}, $I_{1,1}I_{1,2}$
is the number of ways of assigning $2cn$ balls ($cn$ edges) to the vertex subsets $V_i$, $1\le i\le 4$, such that for each vertex subset $V_i$,
there are $2nz_{i,i}$ balls for generating $nz_{i,i}$ edges inside, and another $nz_{i,j}$ ($nz_{j,i}$ if $j<i$), balls for generating $nz_{i,j}$ edges crossing $V_i$ and $V_j$. $I_{1,3}$ is the number of matchings for generating the graph with the numbers of edges inside each separate vertex subset and crossing two different vertex subsets as shown in Figure \ref{MCSM}. Since the two cuts $(D_1, D_2)$ and $(D_3, D_4)$ both have cut size $zn$, it implies the constraints in (\ref{y1_y4}).

We claim that $K(\beta n, nz_{1,2},nz_{1,4},nz_{1,3},2nz_{1,1})$ is the probability that for both cut $(D_1, D_2)$ and cut $(D_3, D_4)$, each vertex in $V_1$ satisfies the local optimality condition. Indeed, since we choose $nz_{1,j}$ edges crossing $V_1$ and $V_j$ u.a.r., the resultant joint degree distribution from these edges for the vertices in $V_1$ is the same as the joint distribution of number of balls in the bins when $nz_{1,j}$ balls are thrown u.a.r. into $\beta n$ bins.
Denote this joint degree vector by $(E_i^{(1,j)})_{i\in [\beta n]}, j=2,3,4$. Similarly, let $(E_i^{(1,1)})_{i\in [\beta n]}$ denote the joint degree vector
obtained from $nz_{1,1}$ internal edges.
The local optimality condition of the  cut $(D_1,D_2)$ for part $V_1$ is equivalent to
\begin{align*}
& E_i^{(1,2)}+E_i^{(1,4)}\ge E_i^{(1,1)}+E_i^{(1,3)},
\end{align*}
for each $i\in V_1$. Similarly the local optimality condition of the  cut $(D_3,D_4)$ for part $V_1$ is equivalent to
\begin{align*}
E_i^{(1,3)}+E_i^{(1,4)}\ge E_i^{(1,1)}+E_i^{(1,2)}.
\end{align*}
These two constraints put together are equivalent to the constraint $E_i^{(1,4)}-E_i^{(1,1)} \geq \lvert E_i^{(1,2)}-E_i^{(1,3)}\rvert, i\in V_1$.
Hence, as claimed
\begin{align}\label{eq:K1}
K(\beta n, nz_{1,2},nz_{1,4},nz_{1,3},2nz_{1,1}) =
\mathbb{P}[E_i^{(1,4)}-E_i^{(1,1)} \geq \lvert E_i^{(1,2)}-E_i^{(1,3)}\rvert, i\in V_1]
\end{align}
is the probability that the local optimality condition for each vertex in $V_1$ is satisfied. Likewise the other three terms following $K(\beta n, nz_{1,2},nz_{1,4},nz_{1,3},2nz_{1,1})$ account for the probability that the local optimality conditions are satisfied for the vertices in $V_2$, $V_3$ and $V_4$, respectively, and we complete the proof.
\end{proof}
\begin{remark}
Without loss of generality, we may only consider $z_{j,k}$, $1\leq j \leq k \leq 4$, satisfying
\begin{align}\label{z_ij_cons}
z_{1,4}-2z_{1,1} \geq | z_{1,2}-z_{1,3}|, \quad  z_{2,3}-2z_{2,2} \geq | z_{1,2}-z_{2,4}| \nonumber \\
z_{2,3}-2z_{3,3} \geq | z_{1,3}-z_{3,4}|, \quad  z_{1,4}-2z_{4,4} \geq | z_{2,4}-z_{3,4}|,
\end{align}
since otherwise the corresponding terms $K(\cdot)$ in the definition of $I_2$ corresponds to zero probability event,
which has no contribution to $\mathbb{E}[X^2(zn)]$ in (\ref{EXm}).
\end{remark}

As in the case of first moment argument, our next goal is to obtain bounds on limits of $n^{-1}\log I_1, n^{-1}\log I_2$ for each choice of $\beta,z_{i,j}$.
We note that we may assume $\beta>0$.
Indeed $\beta=0$ corresponds to the two cuts being identical. The corresponding limit $n^{-1}\log I_1I_2$ in this case is
simply $\lim_n n^{-1}\log \E[X(zn)]\le \lim_n n^{-1}\log \E^2[X(zn)]$, since $x$ is assumed to be below $x_u$.

Next we further simplify $I_1$. Since the total number of edges is $cn$, we have $\sum_{1\leq j \leq k \leq 4} z_{j,k}=c$.
Using Stirling's approximation, the terms in $I_1$ are simplified as follows:
\begin{align}
\label{comp1}
\frac{1}{n}\log {n \choose {\beta n} , {(1/2-\beta)n}, {(1/2-\beta)n}, {\beta n}} = -2\beta \log \beta-2(1/2-\beta)\log(1/2-\beta) + o(1).
\end{align}
Further,
\begin{align}
\frac{1}{n} \log I_{1,1}
 &= 2c \left( -\sum_{i=1}^4 \frac{z_{i,i}}{c} \log \frac{z_{i,i}}{c} - 2\sum_{1\leq j < k \leq 4}\frac{z_{j,k}}{2c} \log \frac{z_{j,k}}{2c}   \right)+o(1) \nonumber \\
\label{comp2}
& = -2\sum_{1\leq j \leq k \leq 4} z_{j,k}\log z_{j,k}+2c \log c + 2(c-\sum_{i=1}^4 z_{i,i}) \log 2+o(1), \\
\frac{1}{n} \log I_{1,2}
 &= (2z_{1,1}+2z_{4,4}+2z_{1,4}+z_{1,2}+z_{1,3}+z_{2,4}+z_{3,4})(\log \beta + \log n)  \nonumber \\
    & \quad + (2z_{2,2}+2z_{3,3}+2z_{2,3}+z_{1,2}+z_{2,4}+z_{1,3}+z_{3,4})(\log (1/2-\beta) + \log n) \nonumber \\
& = 2c \log n + \sum_{1 \leq j \leq k \leq 4} z_{j,k} \log \beta_{j,k} + \sum_{i=1}^4 z_{i,i} \log 2+o(1). \label{comp4}
\end{align}
Next
\begin{align}
\frac{1}{n} \log \left( \prod_{1\leq i <j \leq 4} (nz_{i,j})! \right)&= \sum_{1\leq j<k \leq 4} z_{j,k}\log z_{j,k}+\sum_{1\leq j<k \leq 4} z_{j,k}\log n - \sum_{1\leq j<k \leq 4} z_{j,k} +o(1), \\
\frac{1}{n}  \log \prod_{i=1}^4 F(2nz_{i,i})&=\sum_{i=1}^4 2z_{i,i} \frac{\log(2z_{i,i}n)-1}{2} +o(1)  \notag \\
&=\sum_{i=1}^4 z_{i,i}(\log 2 + \log z_{i,i}+\log n-1)+o(1), \label{comp5}\\
\frac{1}{n}  \log \left( n^{2cn} F(2cn) \right)&=2c\log n + 2c \frac{\log(2cn)-1}{2} \notag \\
&=3c\log n + c\log c + c\log 2-c+o(1). \label{comp6}
\end{align}
Using (\ref{comp1})--(\ref{comp6}), we have
\begin{align}
\label{I_1}
\frac{1}{n} \log I_1= -2\beta \log \beta-(1-2\beta)\log(1/2-\beta)+\sum_{1\leq j \leq k \leq 4} z_{j,k} \log \left( \frac{\beta_{j,k}}{z_{j,k}} \right)+c\log(2c)+o(1)
\end{align}
Introduce $\eta_{j,k}$ through the identities
\begin{align}\label{eq:eta}
z_{j,k}&=2\beta_{j,k}c+2\eta_{j,k}\sqrt{c}, \qquad 1\le j\le k\le 4,
\end{align}
and let $\eta=(\eta_{j,k}, 1\le j,k\le 4)$.

\subsection{Bounds on $I_1$}\label{subsection:bounds on I1}
In terms of our notations (\ref{eq:eta}), we have
\begin{align} \label{I_12}
\sum_{1\leq j \leq k \leq 4} z_{j,k} &\log \left( \frac{\beta_{j,k}}{z_{j,k}} \right)+c\log(2c) \nonumber \\
=&-c\sum_{1\le j\le k\le 4}(2\beta_{j,k}+2\eta_{j,k} c^{-{1\over2}})\log(2c+2\eta_{j,k}\beta_{j,k}^{-1}c^{{1\over2}})+c \log (2c)\nonumber \\
=& -2c\sum_{1\le j\le k\le 4}(\beta_{j,k}+\eta_{j,k} c^{-{1\over2}})\log(1+\eta_{j,k}\beta_{j,k}^{-1}c^{-{1\over2}}).
\end{align}
The constraints (\ref{y1_y4}) can be rewritten as
\begin{align}
& \eta_{1,2}+\eta_{1,4}+\eta_{2,3}+\eta_{3,4}=x/2, \label{y1_y4.C1}\\
& \eta_{1,3}+\eta_{1,4}+\eta_{2,3}+\eta_{2,4}=x/2, \label{y1_y4.C2}\\
& \sum_{1\le j\le k\le 4} \eta_{j,k}=0. \label{y1_y4.CSum}
\end{align}
Let
\begin{align*}
I_3(\beta,\eta)=\sum_{1\le j\le k\le 4}(\beta_{j,k}+\eta_{j,k} c^{-{1\over2}})\log(1+\eta_{j,k}\beta_{j,k}^{-1}c^{-{1\over2}}).
\end{align*}
Thus our next goal is to solve the optimization problem $\min_\eta I_3(\beta,\eta)$ subject to (\ref{y1_y4.CSum}).
The next lemma is used to show that in solving this optimization problem we may restrict the range of $\eta$ to a bound independent of $c$.
As a result we will be able to replace $\log(1+\eta_{j,k}\beta_{j,k}^{-1}c^{-{1\over2}})$ with its Taylor approximation
$\eta_{j,k}\beta_{j,k}^{-1}c^{-{1\over2}}-(1/2)\eta_{j,k}^2\beta_{j,k}^{-2}c^{-1}$.

\begin{lemma}\label{lemma:boundI1}
For every $a>0$ and $\beta\in (0,1/4)$, there exists $c_0=c_0(\beta,a)$ such that for all $c>c_0$ and all $\eta$ satisfying constraint (\ref{y1_y4.CSum}) and
$\|\eta\|_2\ge a$, the following bound holds:
\begin{align*}
I_3(\beta,\eta)\ge {a^2\over 4c}.
\end{align*}
\end{lemma}

\begin{proof}
$I_3$ is a convex function in $\eta$, taking value zero at $\eta=0$. Thus
\begin{align}\label{eq:LinearizedI3}
{a\over \|\eta\|_2}I_3(\beta,\eta)\ge I_3(\beta,{a\over \|\eta\|_2}\eta).
\end{align}
Using Taylor expansion $\log(1+b)=b-b^2/2+O(b^3)$ for some constant $b$ with $|b|<1$, we can find $c_1=c_1(\beta,a)$ large enough so that for all $c>c_1$
\begin{align*}
\Big|\log(1+{a \over \|\eta\|_2}\eta_{j,k}\beta_{j,k}^{-1}c^{-{1\over2}})-{a \over \|\eta\|_2}\eta_{j,k}\beta_{j,k}^{-1}c^{-{1\over2}}
+{1\over 2}\left({a \over \|\eta\|_2}|\eta_{j,k}|\beta_{j,k}^{-1}c^{-1/2}\right)^2\Big|\le c^{-{5\over 4}}\frac{\eta_{j,k}^2}{\| \eta \|_2^2},
\end{align*}
where the exponent $5/4$ is chosen somewhat arbitrary, and any exponent strictly larger than $1$ and less than $3/2$ can serve our purpose.
Thus, since
\begin{align*}
I_3(\beta,{a\over \|\eta\|_2}\eta)
=\sum_{1\le j\le k\le 4}(\beta_{j,k}+{a \over \|\eta\|_2}\eta_{j,k}c^{-{1\over2}})\log(1+{a \over \|\eta\|_2}\eta_{j,k}\beta_{j,k}^{-1}c^{-{1\over2}}),
\end{align*}
then
\begin{align*}
\Big|I_3(\beta,{a \over \|\eta\|_2}\eta)
&-\sum_{1\le j\le k\le 4}(\beta_{j,k}+{a\over \|\eta\|_2}\eta_{j,k} c^{-{1\over2}}){a\over \|\eta\|_2}\eta_{j,k}\beta_{j,k}^{-1}c^{-{1\over2}}\\
&+{1\over 2}\sum_{1\le j\le k\le 4}(\beta_{j,k}+{a\over \|\eta\|_2}\eta_{j,k} c^{-{1\over2}})\left({a\over \|\eta\|_2}\eta_{j,k}\beta_{j,k}^{-1}c^{-{1\over2}}\right)^2\Big| \\
&\le c^{-{5\over 4}}\sum_{1\le j\le k\le 4}(\beta_{j,k}+{a\over \|\eta\|_2}|\eta_{j,k}| c^{-{1\over2}})\frac{\eta_{j,k}^2}{\| \eta \|_2^2}. \\
\end{align*}
We can find $c_2=c_2(\beta,a)$ sufficiently large so that the expression on the right-hand side is at most
\begin{align*}
\le {1\over 8}{a^2 \over c}\sum_{1\le j\le k\le 4}\beta_{j,k}^{-1} \frac{\eta_{j,k}^2}{\| \eta \|_2^2}
\end{align*}
for all $c>c_2$.
On the other hand, applying constraint (\ref{y1_y4.CSum})
\begin{align*}
\sum_{1\le j\le k\le 4}(\beta_{j,k}&+{a\over \|\eta\|_2}\eta_{j,k} c^{-{1\over2}}){a\over \|\eta\|_2}\eta_{j,k}\beta_{j,k}^{-1}c^{-{1\over2}}
-{1\over 2}\sum_{1\le j\le k\le 4}(\beta_{j,k}+{a\over \|\eta\|_2}\eta_{j,k} c^{-{1\over2}})\left({a\over \|\eta\|_2}\eta_{j,k}\beta_{j,k}^{-1}c^{-{1\over2}}\right)^2 \\
&={a^2\over 2c}\sum_{1\le j\le k\le 4} \beta_{j,k}^{-1} \frac{\eta_{j,k}^2}{\| \eta \|_2^2}
-{1\over 2}\sum_{1\le j\le k\le 4}{a\over \|\eta\|_2}\eta_{j,k} c^{-{1\over2}}\left({a\over \|\eta\|_2}\eta_{j,k}\beta_{j,k}^{-1}c^{-{1\over2}}\right)^2
\end{align*}
We can find $c_3=c_3(\beta,a)$ sufficiently large so that the second term in the expression above is also at most
\begin{align*}
\le {1\over 8}{a^2 \over c}\sum_{1\le j\le k\le 4} \frac{\eta_{j,k}^2}{\| \eta \|_2^2} \beta_{j,k}^{-1}
\end{align*}
in absolute value for all $c>c_3$.

We obtain
\begin{align*}
I_3(\beta,{a\over \|\eta\|_2}\eta)
\ge {a^2\over 2c}\sum_{1\le j\le k\le 4} \frac{\eta_{j,k}^2}{\| \eta \|_2^2} \beta_{j,k}^{-1}-2{a^2\over 8c}\sum_{1\le j\le k\le 4} \frac{\eta_{j,k}^2}{\| \eta \|_2^2} \beta_{j,k}^{-1}
={a^2\over 4c}\sum_{1\le j\le k\le 4} \frac{\eta_{j,k}^2}{\| \eta \|_2^2} \beta_{j,k}^{-1}.
\end{align*}
Combining with (\ref{eq:LinearizedI3}) we conclude that for all $c>c_0\triangleq\max(c_1,c_2,c_3)$
\begin{align*}
I_3(\beta,\eta)&\ge {\|\eta\|_2a\over 4c}\sum_{1\le j\le k\le 4}\frac{\eta_{j,k}^2}{\| \eta \|_2^2} \beta_{j,k}^{-1} \ge {a^2\over 4c}.
\end{align*}
where we have used the fact that $\beta_{j,k} \in (0,1)$ for $1 \leq j \leq k \leq 4$ and $\| \eta \|_2 \geq a$.
\end{proof}

\subsection{Bounds on $I_2$}\label{subsection:bounds on I2}
Suppose $\mu_j=\mu_j(n,c)$ are positive integer sequences such that the limits $\lambda_j=\lim_n\mu_j/n$ exists, take rational values
and satisfy $\lambda_j=\Theta_c(c), \lambda_4-\lambda_1=O_c(\sqrt{c}), \lambda_2-\lambda_3=O_c(\sqrt{c})$.
Suppose further that the following limits exist:
\begin{align}
\label{b1b3}
b_1=\lim_{c\rightarrow \infty} \frac{\lambda_2-\lambda_3}{\sqrt{\lambda_2+\lambda_3}}, \; b_2=\lim_{c \rightarrow \infty} \sqrt{\frac{\lambda_2+\lambda_3}{\lambda_4+\lambda_1}}, \; b_3=\lim_{c \rightarrow \infty} \frac{\lambda_4-\lambda_1}{\sqrt{\lambda_4+\lambda_1}}
\end{align}
Let $(B_i^{(j)})_{i\in [n]}$, $j=1,\dots,4$, be four families of i.i.d. Poisson random variables with mean $\lambda_{j}$.
Given independent standard normal random variables $Z_1,Z_2$, let
\begin{align*}
P_2=\pr\left(Z_1+b_3 \geq b_2 \left\lvert Z_2+b_1 \right\rvert \right),
\end{align*}
and let
\begin{align}
P(\theta_1,\theta_2,b_1,b_2,b_3)&=\E[\exp(\theta_1Z_1+\theta_2Z_2)\mathbb{I}(Z_1+b_3 \geq b_2 \left\lvert Z_2+b_1 \right\rvert)] \notag\\
&={1\over 2\pi}\int_{t_1+b_3 \geq b_2 \left\lvert t_2+b_1 \right\rvert}\exp\left(\theta_1 t_1+\theta_2 t_2-{t_1^2+t_2^2\over 2}\right)dt_1dt_2.\notag    
\end{align}

\begin{lemma}\label{compute_K4}
The following large deviations limit exists
\begin{align}\label{eq:I4}
I_4\triangleq \lim_n\frac{1}{n} \log \mathbb{P} \left[ \sum_{i=1}^n B_{i}^{(j)}=\mu_j, j=1,\dots,4 \;
\middle | \; B_i^{(4)}-B_i^{(1)} \geq \lvert B_i^{(2)}-B_i^{(3)} \rvert, i \in [n] \right],
\end{align}
and satisfies

\begin{align}
I_4
\label{LPD_4d_ch1}
& = -\log P_2+\inf_{\theta_1,\theta_2}\log P(\theta_1,\theta_2,b_1,b_2,b_3) + o_c(1).
\end{align}
Furthermore $\inf_{\theta_1,\theta_2}\log P(\theta_1,\theta_2,b_1,b_2,b_3)$ is achieved at a unique point
$\theta_1^*,\theta_2^*$ which is also the unique solution to the system of equations
\begin{align} \label{unique_theta12}
\frac{\partial \log P(\theta_1,\theta_2,b_1,b_2,b_3)} {\partial \theta_1}=0, \quad \frac{\partial \log P(\theta_1,\theta_2,b_1,b_2,b_3)}{ \partial \theta_2}=0
\end{align}
\end{lemma}

\begin{proof}
Let $X_i^{(j)}=\frac{B_i^{(j)}-\lambda_j}{\sqrt{\lambda_j}}$, $i=1,\dots,n$ and $j=1,2,3,4$. The probability in (\ref{eq:I4}) is rewritten by
\begin{align}
\label{lpd_4d}
\mathbb{P} \left[\sum_{i=1}^n X_i^{(j)}=0, j=1,\dots,4 \middle| \mathcal{E}_i , i\in[n]  \right]
\end{align}
where by substituting $X_i^{(j)}$ for $B_i^{(j)}$ the event $\{ B_i^{(4)}-B_i^{(1)} \geq \lvert B_i^{(2)}-B_i^{(3)} \rvert \}$ is equivalent to
$$
\mathcal{E}_i=\left\{\frac{\sqrt{\lambda_4}X_i^{(4)}-\sqrt{\lambda_1}X_i^{(1)}}{\sqrt{\lambda_4+\lambda_1}}+\frac{\lambda_4-\lambda_1}{\sqrt{\lambda_4+\lambda_1}} \geq \sqrt{\frac{\lambda_2+\lambda_3}{\lambda_4+\lambda_1}} \left\lvert \frac{\sqrt{\lambda_2}X_i^{(2)}-\sqrt{\lambda_3}X_i^{(3)}}{\sqrt{\lambda_2+\lambda_3}}+\frac{\lambda_2-\lambda_3}{\sqrt{\lambda_2+\lambda_3}}  \right\rvert \right\}.
$$
Applying Theorem \ref{LCLTLDP} yields
$$
\lim_{n\rightarrow \infty} \frac{1}{n} \log \mathbb{P} \left[\sum_{i=1}^n X_i^{(j)}=0, j=1,\dots,4 \middle| \mathcal{E}_i , i\in[n]  \right]= -I(0,0,0,0)
$$
where $\theta=(\theta_1,\ldots,\theta_4)$, $I(x_1,x_2,x_3,x_4)=\sup_{\theta \in \mathbb{R}^4}(\sum_{i=1}^4 \theta_i x_i-\log(M_c(\theta)))$
and and $M_c(\theta)$ are respectively the large deviations rate function and  the MGF of $(X_1^{(1)},X_1^{(2)},X_1^{(3)},X_1^{(4)})$
conditional on $\mathcal{E}_1$. Specifically,
\begin{align*}
M_c(\theta)=\E[\exp(\sum_{1\le j\le 4}\theta_j X_1^{(j)})| \mathcal{E}_1].
\end{align*}
Since $X_1^j$ converges in distribution to a standard normal random variable and  $\lambda_2/(\lambda_2+\lambda_4),\lambda_1/(\lambda_1+\lambda_3)$
converge to $1/2$ as $c\rightarrow\infty$, then  $\pr(\mathcal{E}_1)$ converges to the probability of the event
\begin{align*}
\mathcal{D}_1\triangleq\biggl\{\frac{Z_4}{\sqrt{2}}-\frac{Z_1}{\sqrt{2}}+b_3 \geq b_2 \left\lvert \frac{Z_2}{\sqrt{2}}-\frac{Z_3}{\sqrt{2}}+b_1 \right\rvert \biggr\},
\end{align*}
where $Z_j, 1\le j\le 4$ are four independent standard normal random variables. We recognize $\pr(\mathcal{D}_1)$ as $P_2$ since
$\frac{Z_2}{\sqrt{2}}-\frac{Z_4}{\sqrt{2}}$ and $\frac{Z_1}{\sqrt{2}}-\frac{Z_3}{\sqrt{2}}$ are independent standard normal random variables.
Thus as $c\rightarrow\infty$,
\begin{align*}
M_c(\theta) &\rightarrow
\mathbb{E}[\exp(\sum_{i=1}^4 \theta_i Z_1^{(i)}) \mid \mathcal{D}_1] \\
&=\frac{1}{P_2} \int_{\mathcal{D}_1} \frac{1}{(2\pi)^2}\exp \left(\sum_{i=1}^4 \left(\theta_i t_i-\frac{t_i^2}{2}\right)  \right)dt_1 \cdots dt_4\\
&\triangleq M(\theta),
\end{align*}
where for convenience we also let
\begin{align*}
\mathcal{D}_1=\biggl\{ (t_1,\ldots,t_4)\in\R^4:\frac{t_2}{\sqrt{2}}-\frac{t_4}{\sqrt{2}}+b_3 \geq b_2 \left\lvert \frac{t_1}{\sqrt{2}}-\frac{t_3}{\sqrt{2}}+b_1 \right\rvert \biggr\}.
\end{align*}
Thus from this point we focus on the optimization problem
\begin{align*}
\sup_{\theta}\left(-\log M(\theta)\right)=-\inf_{\theta}\log M(\theta).
\end{align*}
We again use the fact that since $M(\theta)$ is the MGF which is finite for all $\theta$, then $\log M(\theta)$ is strictly convex and the unique optimal solution
is achieved at a unique point $\theta^*$ where the gradient vanishes. Thus the defining identities for $\theta^*$ are
\begin{align}\label{eq:Optimaltheta}
{\partial \log M(\theta) \over \theta_j}\Big|_{\theta=\theta^*}&=
M^{-1}(\theta^*) \int_{\mathcal{D}_1} \frac{t_j}{(2\pi)^2}\exp \left(\sum_{i=1}^4 \left(\theta_i^* t_i-\frac{t_i^2}{2}\right)  \right)dt_1 \cdots dt_4 \\
&=0, \notag
\end{align}
for $j=1,\ldots,4$, where we use the fact that $P_2$ does not depend on $\theta$ and thus disappears in the gradient.

Now we take advantage of a certain symmetry of $\mathcal{D}_1$. Note that $(t_1,t_2,t_3,t_4)\in\mathcal{D}_1$ iff
$(-t_3,-t_4,-t_1,-t_2)\in\mathcal{D}_1$. This implies that $(\theta_1^*,\theta_2^*,\theta_3^*,\theta_4^*)$ solves (\ref{eq:Optimaltheta})
iff so does $(-\theta_3^*,-\theta_4^*,-\theta_1^*,-\theta_2^*)$. The uniqueness of the optimal solution implies that $\theta_3^*=-\theta_1^*$ and $\theta_4^*=-\theta_2^*$.
In this case again since $(Z_2-Z_4)/\sqrt{2}$ and $(Z_1-Z_3)/\sqrt{2}$ are standard normal when $Z_1,\ldots,Z_4$ are independent standard normal, we recognize
$M(\theta^*)$ as
\begin{align*}
\E[\exp(\theta_1^* Z_1+\theta_2^* Z_2)|Z_1+b_3\ge b_2|Z_2+b_1|]
&=P_2^{-1}\int_{t_1+b_3\ge b_2|t_2+b_1|}{1\over 2\pi}\exp\left(\theta_1^*t_1+\theta_2^*t_2-{t_1^2+t_2^2\over 2}\right)dt_1dt_2.
\end{align*}
We recognize this expression as $P(\theta_1,\theta_2,b_1,b_2,b_3)/P_2$. Hence, (\ref{unique_theta12}) follows from (\ref{eq:Optimaltheta}). This completes the proof.
\end{proof}

We now establish certain properties of the function $P(\theta_1,\theta_2,b_1,b_2,b_3)$. For $b_1=0$, we also give the characterization of $\theta_1$ and $\theta_2$ which obtain the infimum of $\log P(\theta_1,\theta_2,0,b_2,b_3)$ over $\theta_1$ and $\theta_2$.
\begin{lemma}\label{lemma:theta1b1}
The following inequality holds for every $\theta_1,b_1,b_3\in\R$ and $b_2\in\R_+$:
\begin{align}\label{eq:theta1b1}
& \inf_{\theta_2}  P(\theta_1,\theta_2,b_1,b_2,b_3) \leq \inf_{\theta_2} \log P(\theta_1,\theta_2,0,b_2,b_3)=P(\theta_1,0,0,b_2,b_3).
\end{align}
Furthermore, $\log P(\theta_1,0,0,b_2,b_3)$ is a concave function in $b_3$ for every $\theta_1$, and hence $\inf_{\theta_1} \log P(\theta_1,0,0,b_2,b_3)$ is also a concave function in $b_3$. Finally, $\theta_1^*$ defined by $\theta_1^*=\mathrm{arginf}_{\theta_1} \log P(\theta_1,0,0,b_2,b_3)$ is the unique solution to
\begin{align} \label{unique_theta1}
\theta_1+\frac{1}{Q(\theta_1,b_2,b_3)}\int_0^{\infty}  \exp \left(-\frac{t_2^2+(b_2 t_2-\theta_1-b_3)^2}{2} \right) dt_2=0
\end{align}
where $Q(\theta_1,b_2,b_3)$ is given in (\ref{Q_exp}).
\end{lemma}

\begin{proof}
First, we claim that $P(\theta_1,\theta_2,b_1,b_2,b_3) \le P(\theta_1,\theta_2,-\theta_2,b_2,b_3)$. In order to show this, we rewrite $P(\theta_1,\theta_2,b_1,b_2,b_3)$ in another form. We use the change of variables $z_1=t_1-\theta_1$ and $z_2=t_2-\theta_2$
\begin{align*}
&P(\theta_1,\theta_2,b_1,b_2,b_3) \\
&=\frac{1}{2\pi} \exp((\theta_1^2+\theta_2^2)/2) \int_{z_1+\theta_1+b_3\ge b_2|z_2+\theta_2+b_1|} \exp(-(z_1^2+z_2^2)/2)dz_1 dz_2  \\
&=\frac{1}{2\pi} \exp((\theta_1^2+\theta_2^2)/2) \bigg(\int_{-\theta_2-b_1}^{\infty} dz_2 \int_{b_2(z_2+\theta_2+b_1)-\theta_1-b_3}^{\infty} \exp(-(z_1^2+z_2^2)/2) dz_1  \\
& \quad  + \int_{-\infty}^{-\theta_2-b_1} dz_2 \int_{-b_2(z_2+\theta_2+b_1)-\theta_1-b_3}^{\infty} \exp(-(z_1^2+z_2^2)/2) dz_1 \bigg)
\end{align*}
We use the change of variables $t_1=z_1+\theta_1+b_3$ and $t_2=z_2+\theta_2+b_1$ for the first integral, and the change of variables $t_1=z_1+\theta_1+b_3$ and $t_2=-z_2-\theta_2-b_1$ for the second integral. The integral above is
\begin{align*}
\frac{1}{2\pi} \exp((\theta_1^2+\theta_2^2)/2) \int_{0}^{\infty} dt_2 \int_{b_2 t_2}^{\infty} & \Big( \exp(-((t_1-\theta_1-b_3)^2+(t_2-\theta_2-b_1)^2)/2) \\
                                                                                              & +\exp(-((t_1-\theta_1-b_3)^2+(t_2+\theta_2+b_1)^2)/2) \Big) dt_1
\end{align*}
which we rewrite as
\begin{align} \label{P_expand}
\frac{1}{2\pi} \exp((\theta_1^2+\theta_2^2)/2) \int_{0}^{\infty} \exp(-(t_1-\theta_1-b_3)^2/2)dt_1 \int_{0}^{t_1/b_2} & \Big( \exp(-(t_2-\theta_2-b_1)^2/2)  \nonumber \\
                                                                                              & +\exp(-(t_2+\theta_2+b_1)^2/2) \Big) dt_2
\end{align}
Its partial derivative with respect to $b_1$ gives
\begin{align*}
&{\partial  P(\theta_1,\theta_2,b_1,b_2,b_3) \over \partial b_1} \\
&=\frac{1}{2\pi} \exp((\theta_1^2+\theta_2^2)/2) \int_{0}^{\infty} \exp \left(- \left((t_1-\theta_1-b_3)^2+(t_1/b_2-\theta_2-b_1)^2 \right)/2 \right)\big(\exp(-2(\theta_2+b_1)t_1/b_2 )-1 \big) dt_1
\end{align*}
Notice that $b_2 \in \mathbb{R}_{+}$, it is easy to see that the derivative above is positive when $b_1<-\theta_2$, zero when $b_1=-\theta_2$, and negative when $b_1>-\theta_2$. Hence, $b_1=-\theta_2$ maximizes $P(\theta_1,\theta_2,b_1,b_2,b_3)$, i.e.
\begin{align*}
P(\theta_1,\theta_2,b_1,b_2,b_3) \leq P(\theta_1,\theta_2,-\theta_2,b_2,b_3)
\end{align*}
When $b_1=-\theta_2$, from (\ref{P_expand}) it is easy to see that $\theta_2=0$ minimizes $P(\theta_1,\theta_2,-\theta_2,b_2,b_3)$, which, together with the inequality above, gives
$$
\inf_{\theta_2} P(\theta_1,\theta_2,b_1,b_2,b_3) \leq \inf_{\theta_2} P(\theta_1,\theta_2,-\theta_2,b_2,b_3)= P(\theta_1,0,0,b_2,b_3),
$$
and hence the first inequality in (\ref{eq:theta1b1}) follows.

Inherited from the strict convexity of $\log M(\theta)$, $\log P(\theta_1,\theta_2,0,b_2,b_3)$ is also strictly convex in $\theta_2$. Also, (\ref{P_expand}) implies that for $b_1=0$, $P(\theta_1,\theta_2,0,b_2,b_3)$ is symmetric about $\theta_2=0$, and thus the derivative of $\log P(\theta_1,\theta_2,0,b_2,b_3)$ with respect to $\theta_2$ is always $0$ at $\theta_2=0$. These two facts establish the equality in (\ref{eq:theta1b1}).

Next, we prove the concavity of $\inf_{\theta_1}\log P(\theta_1,0,0,b_2,b_3)$ in $b_3$. Let $f(x,y)=\exp(-(x^2+y^2)/2)$, which is log-concave. After changing the order of integration in (\ref{P_expand}), we have
\begin{align}
\label{P_00}
P(\theta_1,0,0,b_2,b_3)= &\frac{1}{\pi} \exp(\theta_1^2/2)\int_{0}^{\infty} \int_{b_2 t_2}^{\infty} \exp(-((t_1-\theta_1-b_3)^2+t_2^2)/2) dt_1 dt_2  \\
                       = &\frac{1}{\pi} \exp(\theta_1^2/2) \int_{0}^{\infty} \int_{b_2 t_2}^{\infty} f(t_1-\theta_1-b_3,t_2) dt_1 dt_2 \nonumber
\end{align}
The log-concavity of $f(x,y)$ implies that for all $\theta_1 \in \mathbb{R}$, $b_3^{(1)}$, $b_3^{(2)} \in \mathbb{R}$, $(t_1,t_2)$, $(\bar{t}_1,\bar{t}_2) \in \mathbb{R}^2$ and $\lambda \in (0,1)$, we have
\begin{align}
\label{log_concave_f2}
& f(\lambda(t_1-b_3^{(1)}-\theta_1)+(1-\lambda)(\bar{t}_1-b_3^{(2)}-\theta_1), \lambda t_2 + (1-\lambda)\bar{t}_2) \nonumber \\
& \geq f^{\lambda}(t_1-b_3^{(1)}-\theta_1,t_2)f^{1-\lambda}(\bar{t}_1-b_3^{(2)}-\theta_1,\bar{t}_2)
\end{align}
Let $\mathcal{S}=\{(t_1,t_2) \in \mathbb{R}^2: t_1 \geq b_2 t_2,  t_2 \geq 0 \}$, and
\begin{align}
\label{hg1_g22}
& h(t_1,t_2) = f(t_1-(\lambda b_3^{(1)} + (1-\lambda)b_3^{(2)})-\theta_1, t_2) \mathbf{1}_{\mathcal{S}}(t_1,t_2) \nonumber \\
& 	g_1(t_1,t_2) = f(t_1-b_3^{(1)}-\theta_1,t_2) \mathbf{1}_{\mathcal{S}}(t_1,t_2) \nonumber \\
& g_2(t_1,t_2) = f(t_1-b_3^{(2)}-\theta_1, t_2) \mathbf{1}_{\mathcal{S}}(t_1,t_2) .
\end{align}
where $\mathbf{1}_{\mathcal{S}}(t_1,t_2)$ is the indicator function. $h(t_1,t_2)$, $g_1(t_1,t_2)$ and $g_2(t_1,t_2)$ are non-negative functions, which by (\ref{log_concave_f2}) satisfy
$$
h(\lambda(t_1,t_2)+(1-\lambda)(\bar{t}_1,\bar{t}_2)) \geq g_1^{\lambda}(t_1, t_2)g_2^{1-\lambda}(\bar{t}_1, \bar{t}_2)
$$
for any $(t_1,t_2), (\bar{t}_1,\bar{t}_2) \in \mathbb{R}^2$. Then Pr\'{e}kopa--Leindler inequality, together with an argument similar to the one for the proof of Lemma \ref{maximize_K2d}, yields that $\log P(\theta_1,0,0,b_2,b_3)$ is a concave function for each $\theta_1$, and further that $\inf_{\theta_1}\log P(\theta_1,0,0,b_2,b_3)$ is also a concave function in $b_3$.

(\ref{eq:theta1b1}) and Lemma \ref{compute_K4} yields that the $\theta_1^*$ and $\theta_2^{*}$ which obtains the infimum of $\inf_{\theta_1,\theta_2} P(\theta_1,\theta_2,0,b_2,b_3)$ is $\theta_2^*=0$ and $\theta_1^*$ is uniquely determined by setting the derivative of $\log P(\theta_1,0,0,b_2,b_3)$ with respect to $\theta_1$ to be $0$. Recall that $P(\theta_1,0,0,b_2,b_3)$ is given in (\ref{P_00}),
\begin{align*}
&\frac{\partial \log P(\theta_1,0,0,b_2,b_3)}{\partial \theta_1} \\
&=\theta_1 + \frac{1}{Q(\theta_1,b_2,b_3)} \frac{\partial Q(\theta_1,b_2,b_3)}{\partial \theta_1} \\
&=\theta_1 + \frac{1}{Q(\theta_1,b_2,b_3)} \int_0^{\infty} \int_{b_2 t_2}^{\infty} \exp \left( -\frac{t_2^2+(t_1-\theta_1-b_3)^2}{2}\right)( t_1-\theta_1-b_3) dt_1  dt_2.
\end{align*}
Then (\ref{unique_theta1}) follows from the inner integral over $t_1$.
\end{proof}
For shortness, we write $P(\theta_1,b_2,b_3)$ for $P(\theta_1,0,0,b_2,b_3)$ and we recall that this is definition of $P(\theta_1,b_2,b_3)$ as given in (\ref{eq:Pfunction}).

Recall from the proof of Lemma \ref{lemma:theta1b1}, we restate here for convenience the identity (\ref{P_00})
\begin{align}
\label{P_002}
P(\theta_1,b_2,b_3)= &\frac{1}{\pi} \exp(\theta_1^2/2)\int_{0}^{\infty} \int_{b_2 t_2}^{\infty} \exp(-((t_1-\theta_1-b_3)^2+t_2^2)/2) dt_1 dt_2.
\end{align}

\subsection{Computing the limit of $\lim_n n^{-1}\log (I_1I_2)$.}\label{subsection:optimizingI1I2}
In this subsection, we first use a special setting to claim that the maximum of $\lim_n n^{-1}\log (I_1I_2)$ is obtained when $\|\eta\|\leq a$ for some positive constant $a$,
and then consider the problem of maximizing of $\lim_n n^{-1}\log (I_1I_2)$.

Note that setting $\beta=1/4$, $\eta_{1,4}=\eta_{2,3}=x/4$,  $\eta_{1,2}=\eta_{1,3}=\eta_{2,4}=\eta_{3,4}=0$, and $\eta_{j,j}=-x/8, 1\le j\le 4$,
we obtain $\eta$ which satisfies all constraints (\ref{y1_y4.C1})-(\ref{y1_y4.CSum}) and
\begin{align*}
I_3(1/4,\eta)&=
\sum_{1\le j\le k\le 4}(\beta_{j,k}+\eta_{j,k} c^{-{1\over2}})(\eta_{j,k}\beta_{j,k}^{-1}c^{-{1\over2}}-(1/2)(\eta_{j,k}\beta_{j,k}^{-1}c^{-{1\over2}})^2)+o_c(c^{-1}) \\
&=\sum_{1\leq j \leq k \leq 4} \frac{1}{2}\eta_{j,k}^2 \beta_{j,k}^{-1} c^{-1} + o_c(c^{-1}) = 2x^2/c + o_c(c^{-1}) .
\end{align*}
Substituting $I_3(1/4,\eta)$ back to (\ref{I_12}) and then to (\ref{I_1}), we have that
\begin{align} \label{I_1_beta14}
\lim_{n\rightarrow \infty}\frac{1}{n}\log I_1=2\log 2  - 4x^2 + o_c(1)
\end{align}
We now recall notation (\ref{eq:eta}). Applying (\ref{eq:K1}), Lemma~\ref{lemmaK4d}, Lemma~\ref{compute_K4}, Lemma~\ref{lemma:theta1b1} and canceling $\log P_2$ with $\log \pr(\mathcal{B}_i)$  we obtain
\begin{align*}
&\lim_{n\rightarrow \infty}{1\over n}\log K( n/4, z_{1,2},z_{1,4},z_{1,3},2z_{1,1}) \\
&=\frac{1}{4} \inf_{\theta_1,\theta_2}
\log P\left(\theta_1,\theta_2,{2(\eta_{1,2}-\eta_{1,3})\over \sqrt{2\beta(\beta_{1,2}+\beta_{1,3})}},{\sqrt{\beta_{1,2}+\beta_{1,3}}\over \sqrt{\beta_{1,4}+2\beta_{1,1}}},
{2(\eta_{1,4}-2\eta_{1,1})\over \sqrt{2\beta(\beta_{1,4}+2\beta_{1,1})}}\right)+o_c(1) \\
&=\frac{1}{4} \inf_{\theta_1,\theta_2}
\log P(\theta_1,\theta_2,0,1,4x)+o_c(1) \\
&={1 \over 4} \log P\left(\theta_1^*,0,0,1, 4x \right)+o_c(1).
\end{align*}
where by (\ref{unique_theta1}), $\theta_1^*$ is the unique solution to
\begin{align}\label{theta_beta14}
\theta_1+\frac{1}{Q(\theta_1,1,4x)} \int_0^{\infty} \exp\left( -(t_2^2+(t_2-\theta_1-4x)^2)/2 \right)dt_2=0
\end{align}
where
\begin{align*}
Q(\theta_1,1,4x)=\int_0^{\infty} \int_{t_2}^{\infty} \exp(-((t_1-\theta_1-4x)^2+t_2^2)/2) dt_1 dt_2
\end{align*}
To further simplify $Q(\theta_1,1,4x)$, we introduce the following lemma.
\begin{lemma}
\label{basis_cal}
For $a\in \mathbb{R}$,
\begin{align}
\label{dinte1}
\frac{1}{\pi} \int_{0}^{\infty}  \int_{t_2}^{\infty} \exp \left(-((t_1-a)^2+t_2^2)/2 \right) dt_1 dt_2=(1+\erf(a/2))^2/4
\end{align}
\end{lemma}
\begin{proof}
Using the change of variable $z_1=t_1-a$, the left double integral is
\begin{align}
\label{double_inte2}
\frac{1}{\pi} \int_{0}^{\infty}  \int_{t_2-a}^{\infty} \exp(-(z_1^2+t_2^2)/2) d z_1 dt_2 =& \frac{1}{4}+\frac{1}{\pi} \int_0^{\infty} dt_2 \int_{t_2-a}^{t_2} \exp \left( -(z_1^2+t_2^2)/2  \right) dz_1
\end{align}
Using the transformation $u=(z_1+t_2)/2$ and $v=(-z_1+t_2)/2$ where the determinant of the corresponding Jacobian matrix is $2$ and the integration region is
$$
\{(u,v): 0 \leq u, 0 \leq v \leq a/2\} \cup \{(u,v): -v \leq u \leq 0, 0 \leq v \leq a/2\},
$$
the integral in (\ref{double_inte2}) is
\begin{align}
& \frac{1}{4} + \frac{2}{\pi} \int_0^{\infty} du \int_{0}^{a/2} \exp(-(u^2+v^2))dv + \frac{2}{\pi} \int_0^{a/2} dv \int_{-v}^0 \exp(-(u^2+v^2))du  \nonumber \\
& =  \frac{1}{4} + \frac{1}{\sqrt{\pi}} \int_0^{a/2} \exp(-v^2) dv +   \frac{1}{\pi} \int_{0}^{a/2} dv \int_{0}^{a/2} \exp(-(u^2+v^2)) du \nonumber  \\
& = \frac{1}{4} + \frac{1}{2} \erf(a/2) + \frac{1}{4} (\erf(a/2))^2
\end{align}
and this equals the formula at (\ref{dinte1}).
\end{proof}
We claim that the unique solution $\theta_1^*$ to (\ref{theta_beta14}) is twice of the one to (\ref{transce2}) for the same $x$. By Lemma \ref{basis_cal} and the integral
\begin{align*}
\int_0^{\infty} \exp(-(t_2^2+(t_2-\theta_1-4x)^2)/2) dt_2=\sqrt{\pi}/2\exp(-(\theta_1+4x)^2/4)(1+\erf((\theta_1+4x)/2)),
\end{align*}
(\ref{theta_beta14}) is rewritten by
\begin{align*}
& \theta_1 + \frac{\sqrt{\pi}/2\exp(-(\theta_1+4x)^2/4)(1+\erf((\theta_1+4x)/2))}{\pi (1+\erf((\theta_1+4x)/2))^2/4}=0 \\
& \Rightarrow \theta_1/2+ \sqrt{\frac{1}{\pi}}\frac{e^{-(2x+\theta_1/2)^2}}{1+\erf(2x+\theta_1/2)}=0
\end{align*}
which is (\ref{transce2}) by setting $\theta_1/2$ to $\theta$. Recall that the unique solution to (\ref{transce2}) is $\theta(x)$ for a $x$ satisfying (\ref{eq:xupperrange}). For this special setting, it is easy to see that the $K(\cdot)$ for each vertex subset $V_1,\dots,V_4$ are the same. Hence we have
\begin{align} \label{I_2_beta14}
\frac{1}{n} \log I_2 = 4/n \log K(n/4,nz_{1,2},nz_{1,4},nz_{1,3},2nz_{1,1})=\log P(2\theta(x),0,0,1,4x)
\end{align}
Then by (\ref{I_1_beta14}) and (\ref{I_2_beta14}), we obtain a lower bound for $\sup_{\beta,z_{j,k}}  \lim_{n \rightarrow \infty}\frac{1}{n} \log I_1 I_2$, that is
\begin{align} \label{lb_beta14}
\sup_{\beta,z_{j,k}}  \lim_{n \rightarrow \infty}\frac{1}{n} \log I_1 I_2 \geq 2\log 2-4x^2+\log P(2\theta(x),0,0,1,4x)
\end{align}
For a given $x$ satisfying (\ref{eq:xupperrange}), the right hand side of the last equation is a constant. For a constant $a>0$ such that $\|\eta\|_2 \geq a$, Lemma \ref{lemma:boundI1} yields that $I_3(\beta,\eta) \geq a^2/(4c)$, which imples that
\begin{align*}
\frac{1}{n}\log I_1 \leq -a^2/2 - 2\beta\log \beta-(1-2\beta)\log(1/2-\beta) \leq -a^2/2+2\log 2
\end{align*}
By $\frac{1}{n}\log I_2 \leq 0$, we obtain an upper bound
\begin{align*}
\frac{1}{n} \log I_1 I_2 \leq -a^2/2+2\log 2
\end{align*}
We can increase $a$ such that the upper bound in the last equation is less than the lower bound in (\ref{lb_beta14}) for a given $x$ in (\ref{eq:xupperrange}). Thus when considering the optimization problem of maximizing $I_1 I_2$ over $\beta$ and $\eta$ subject to the constraints (\ref{y1_y4.C1})-(\ref{y1_y4.CSum})
for sufficiently large $c$, we may without the loss of generality consider vectors $\eta$ satisfying $\|\eta\|_2\le a$ for some postive constant $a$. This will be useful
in our later analysis. For vectors $\eta$ satisfying this bound, we obtain an approximation
\begin{align}\label{eq:I3SecondOrder}
I_1(\beta,\eta)=-2\beta\log \beta-(1-2\beta)\log(1/2-\beta)-\sum_{1\le j\le k\le 4}{\eta_{j,k}^2\over \beta_{j,k}}+o_c(1).
\end{align}

We now recall notation (\ref{eq:eta}). Applying (\ref{eq:K1}), Lemma~\ref{lemmaK4d} and Lemma~\ref{compute_K4} and canceling $\log P_2$ with $\log \pr(\mathcal{B}_i)$  we obtain
\begin{align*}
&\lim_n{1\over n}\log K(\beta n, nz_{1,2},nz_{1,4},nz_{1,3},2nz_{1,1}) \\
&=\beta\inf_{\theta_1,\theta_2}
\log P\left(\theta_1,\theta_2,{2(\eta_{1,2}-\eta_{1,3})\over \sqrt{2\beta(\beta_{1,2}+\beta_{1,3})}},{\sqrt{\beta_{1,2}+\beta_{1,3}}\over \sqrt{\beta_{1,4}+2\beta_{1,1}}},
{2(\eta_{1,4}-2\eta_{1,1})\over \sqrt{2\beta(\beta_{1,4}+2\beta_{1,1})}}\right)+o_c(1).
\end{align*}
Similarly
{\small
\begin{align*}
& \lim_n{1\over n}\log K(\beta n, nz_{1,4},nz_{2,4},nz_{3,4},2nz_{4,4}) \\
&=\beta\inf_{\theta_1,\theta_2}
\log P\left(\theta_1,\theta_2,{2(\eta_{3,4}-\eta_{2,4})\over \sqrt{2\beta(\beta_{2,4}+\beta_{3,4})}},{\sqrt{\beta_{2,4}+\beta_{3,4}}\over \sqrt{\beta_{1,4}+2\beta_{4,4}}},
{2(\eta_{1,4}-2\eta_{4,4})\over \sqrt{2 \beta (\beta_{1,4}+2\beta_{4,4}})}\right)+o_c(1), \\
&\lim_n{1\over n}\log K((1/2-\beta) n, nz_{1,2},nz_{2,4},nz_{2,3},2nz_{2,2})=\\
&=(1/2-\beta)\inf_{\theta_1,\theta_2}
\log P\left(\theta_1,\theta_2,{2(\eta_{1,2}-\eta_{2,4})\over \sqrt{2(1/2-\beta)(\beta_{1,2}+\beta_{2,4})}},{\sqrt{\beta_{1,2}+\beta_{2,4}}\over \sqrt{\beta_{2,3}+2\beta_{2,2}}},
{2(\eta_{2,3}-2\eta_{2,2})\over \sqrt{2(1/2-\beta)(\beta_{2,3}+2\beta_{2,2})}}\right)+o_c(1)\\
& \lim_n{1\over n}\log K((1/2-\beta) n, nz_{1,3},nz_{3,4},nz_{2,3},2nz_{3,3})=\\
&=(1/2-\beta)\inf_{\theta_1,\theta_2}
\log P\left(\theta_1,\theta_2,{2(\eta_{1,3}-\eta_{3,4})\over \sqrt{2(1/2-\beta)(\beta_{1,3}+\beta_{3,4})}},{\sqrt{\beta_{1,3}+\beta_{3,4}}\over \sqrt{\beta_{2,3}+2\beta_{3,3}}},
{2(\eta_{2,3}-2\eta_{3,3})\over \sqrt{2(1/2-\beta)(\beta_{2,3}+2\beta_{3,3})}}\right)+o_c(1).
\end{align*}
}
 Combining with (\ref{eq:I3SecondOrder}),  (\ref{I_12}) and (\ref{I_1}), we are thus reduced to solving the optimization problem of maximizing
\begin{align}
&-2\beta \log \beta-(1-2\beta)\log(1/2-\beta)-{\eta_{1,4}^2\over \beta_{1,4}}-{\eta_{2,3}^2\over \beta_{2,3}}
-{\eta_{1,2}^2\over \beta_{1,2}}-{\eta_{1,3}^2\over \beta_{1,3}}-{\eta_{2,4}^2\over \beta_{2,4}}-{\eta_{3,4}^2\over \beta_{3,4}}
-\sum_{j=1}^{4}{\eta_{j,j}^2\over \beta_{j,j}}
\label{eq:QuadraticSum}\\
&+\beta\inf_{\theta_1,\theta_2}
\log P\left(\theta_1,\theta_2,{2(\eta_{1,2}-\eta_{1,3})\over \sqrt{2\beta(\beta_{1,2}+\beta_{1,3})}},{\sqrt{\beta_{1,2}+\beta_{1,3}}\over \sqrt{\beta_{1,4}+2\beta_{1,1}}},
{2(\eta_{1,4}-2\eta_{1,1})\over \sqrt{2\beta(\beta_{1,4}+2\beta_{1,1})}}\right) \nonumber \\
&+\beta\inf_{\theta_1,\theta_2}
\log P\left(\theta_1,\theta_2,{2(\eta_{3,4}-\eta_{2,4})\over \sqrt{2\beta(\beta_{2,4}+\beta_{3,4})}},{\sqrt{\beta_{2,4}+\beta_{3,4}}\over \sqrt{\beta_{1,4}+2\beta_{4,4}}},
{2(\eta_{1,4}-2\eta_{4,4})\over \sqrt{2 \beta (\beta_{1,4}+2\beta_{4,4}})}\right) \notag \\
&+(1/2-\beta)\inf_{\theta_1,\theta_2}
\log P\left(\theta_1,\theta_2,{2(\eta_{1,2}-\eta_{2,4})\over \sqrt{2(1/2-\beta)(\beta_{1,2}+\beta_{2,4})}},{\sqrt{\beta_{1,2}+\beta_{2,4}}\over \sqrt{\beta_{2,3}+2\beta_{2,2}}},
{2(\eta_{2,3}-2\eta_{2,2})\over \sqrt{2(1/2-\beta)(\beta_{2,3}+2\beta_{2,2})}}\right) \notag \\
&+(1/2-\beta)\inf_{\theta_1,\theta_2}
\log P\left(\theta_1,\theta_2,{2(\eta_{1,3}-\eta_{3,4})\over \sqrt{2(1/2-\beta)(\beta_{1,3}+\beta_{3,4})}},{\sqrt{\beta_{1,3}+\beta_{3,4}}\over \sqrt{\beta_{2,3}+2\beta_{3,3}}},
{2(\eta_{2,3}-2\eta_{3,3})\over \sqrt{2(1/2-\beta)(\beta_{2,3}+2\beta_{3,3})}}\right), \notag
\end{align}
subject to (\ref{y1_y4.C1}),(\ref{y1_y4.C2}),(\ref{y1_y4.CSum}).

\begin{lemma}\label{lemma:Optimal eta}
Given $x$ satisfying (\ref{eq:xupperrange}), the value of the optimization problem above equals to the maximum value of the following function in $\beta \in (0, 1/2)$ and $t$:
\begin{align}
\label{minimax}
&-2\beta \log \beta-(1-2\beta)\log(1/2-\beta)-{1\over 2}{t^2\over \beta^2}-{1\over 2}{(x-t)^2\over (1/2-\beta)^2}
+2\beta\inf_{\theta}
\log P\left(\theta,\sqrt{1/2-\beta\over \beta},
{t\over \beta^{3/2}}\right) \nonumber \\
&+2(1/2-\beta)\inf_{\theta}
\log P\left(\theta,\sqrt{\beta\over 1/2-\beta},
{x-t\over (1/2-\beta)^{3/2}}\right).
\end{align}
\end{lemma}

\begin{proof}
Applying Lemma~\ref{lemma:theta1b1} we can set $\theta_2=0$. By the same lemma, the contribution of $\log P$ term is maximized, all else being equal, by
setting $\eta_{1,2}=\eta_{1,3},\eta_{2,4}=\eta_{3,4},\eta_{1,2}=\eta_{2,4},\eta_{1,3}=\eta_{3,4}$ since it makes the third argument of $P$ equal to zero.
At the same time we have $\beta_{i,j}$ take the same value for the corresponding pairs of indices $(1,2),(1,3),(2,4),(3,4)$. Thus replacing
the terms $\eta_{1,2},\eta_{1,3},\eta_{2,4},\eta_{3,4}$ by their average $(\eta_{1,2}+\eta_{1,3}+\eta_{2,4}+\eta_{3,4})/4$ can only increase the quadratic
term in the objective function (\ref{eq:QuadraticSum}). We now analyze how this replacement affects the constraints. From the constraints
(\ref{y1_y4.C1}) and (\ref{y1_y4.C2}) we must have $\eta_{1,2}+\eta_{3,4}=\eta_{1,3}+\eta_{2,4}=x/2-\eta_{1,4}-\eta_{2,3}$.
Thus setting  $\eta_{1,2}=\eta_{1,3}=\eta_{2,4}=\eta_{3,4}$ equal to $(x/2-\eta_{1,4}-\eta_{2,3})/2$ satisfies all of the constraints (\ref{y1_y4.C1})-(\ref{y1_y4.C2}).
We conclude that this substitution
does not decrease the objective function (\ref{eq:QuadraticSum}) and automatically satisfies the constraints (\ref{y1_y4.C1}),(\ref{y1_y4.C2}). In particular,
the constraint (\ref{y1_y4.CSum}) is the only one we should mind.

Next, from Lemma~\ref{lemma:theta1b1} we also have concavity of $\log P$ function in its last argument. Thus replacing $\eta_{1,1}$ and $\eta_{4,4}$ by their average
increases the contribution of the first two $\log P$ terms. At the same time this can only increase the value of the quadratic term in (\ref{eq:QuadraticSum})
since again $\beta_{1,1}=\beta_{4,4}$. A similar observation implies $\eta_{2,2}=\eta_{3,3}$. The constraint (\ref{y1_y4.CSum}) is not affected by this substitution
since $\eta_{j,j}$ appear there only through their sum.

We conclude that the optimization problem is equivalent to maximizing
\begin{align}
&-2\beta \log \beta-(1-2\beta)\log(1/2-\beta)-{\eta_{1,4}^2\over \beta^2}-{\eta_{2,3}^2\over (1/2-\beta)^2}
-{(x/2-\eta_{1,4}-\eta_{2,3})^2\over \beta(1/2-\beta)}
\label{eq:QuadraticSum2}\\
&-4{\eta_{1,1}^2\over \beta^2}-4{\eta_{2,2}^2\over (1/2-\beta)^2}+2\beta\inf_{\theta}
\log P\left(\theta,0,0,\sqrt{1/2-\beta\over \beta},
{\eta_{1,4}-2\eta_{1,1}\over \beta^{3/2}}\right) \label{eq:LD1}\\
&+2(1/2-\beta)\inf_{\theta}
\log P\left(\theta,0,0,\sqrt{\beta\over 1/2-\beta},
{\eta_{2,3}-2\eta_{2,2}\over (1/2-\beta)^{3/2}}\right), \label{eq:LD2}
\end{align}
subject to the only constraint
\begin{align*}
\eta_{1,4}+\eta_{2,3}+2\eta_{1,1}+2\eta_{2,2}+4(x/2-\eta_{1,4}-\eta_{2,3})/2=0,
\end{align*}
which we rewrite as
\begin{align}\label{eq:NewConstraint}
\eta_{1,4}+\eta_{2,3}-2\eta_{1,1}-2\eta_{2,2}=x.
\end{align}
Now we let $t_1=\eta_{1,4}-2\eta_{1,1}$ and $t_2=\eta_{2,3}-2\eta_{2,2}$, allowing us to rewrite the constraint above as
\begin{align}\label{eq:NewConstraint2}
t_1+t_2=x.
\end{align}
Notice that the large deviations terms (\ref{eq:LD1})  and (\ref{eq:LD2}) in the objective function depend on $\eta$ only through $t_1$ and $t_2$. We now
consider unconstrained optimizing the quadratic term (\ref{eq:QuadraticSum2}) in terms of $\eta_{1,4}$ and $\eta_{2,3}$ for a fixed value $t_1$ and $t_2$. The quadratic term is
\begin{align*}
&-{\eta_{1,4}^2\over \beta^2}-{\eta_{2,3}^2\over (1/2-\beta)^2}
-{(x/2-\eta_{1,4}-\eta_{2,3})^2\over \beta(1/2-\beta)}-{(\eta_{1,4}-t_1)^2\over \beta^2}-{(\eta_{2,3}-t_2)^2\over (1/2-\beta)^2}.
\end{align*}
We observe that setting $\eta_{1,4}=t_1/2$ minimizes ${\eta_{1,4}^2\over \beta^2}+{(\eta_{1,4}-t_1)^2\over \beta^2}$. Similar observation applies to setting
$\eta_{2,3}=t_2/2$. At the same time, this setting implies $\eta_{1,4}+\eta_{2,3}=(t_1+t_2)/2=x/2$ and thus nullifies the middle term. We conclude that for a given
$t_1,t_2$ satisfying $t_1+t_2=x$, the optimal value is
\begin{align*}
&-{1\over 2}{t_1^2\over \beta^2}-{1\over 2}{t_2^2\over (1/2-\beta)^2}\\
&+2\beta\inf_{\theta}
\log P\left(\theta,0,0,\sqrt{1/2-\beta\over \beta},
{t_1\over \beta^{3/2}}\right) \\
&+2(1/2-\beta)\inf_{\theta}
\log P\left(\theta,0,0,\sqrt{\beta\over 1/2-\beta},
{t_2\over (1/2-\beta)^{3/2}}\right),
\end{align*}
Setting $t_1=t,t_2=x-t$ completes the proof.
\end{proof}

\ignore{
In light of Lemma~\ref{lemma:Optimal eta} we can further rewrite the optimization problem as
\begin{align*}
&\sup_t W(\beta,x,t),
\end{align*}
where
\begin{align*}
W(\beta,x,y)&\triangleq
-{1\over 4}{y^2\over \beta^2}-{1\over 4}{(2x-y)^2\over (1/2-\beta)^2}
+2\beta\inf_\theta\log P\left(\theta,0,0,\sqrt{1/2-\beta\over \beta},{y\over \sqrt{3/2}\beta}\right) \\
&+2(1/2-\beta)\inf_\theta\log P\left(\theta,0,0,\sqrt{\beta\over 1/2-\beta},{2x-y\over \sqrt{3/2}(1/2-\beta)}\right),
\end{align*}
and we recall
\begin{align*}
P(\theta,0,0,b_2,b_3)&=
{1\over 2\pi}
\int_{t_1+b_3 \geq b_2 \left\lvert t_2 \right\rvert}\exp\left(\theta t_1-{t_1^2+t_2^2\over 2}\right)dt_1dt_2 \\
&={1\over \pi}
\int_{t_1\geq b_2 t_2, t_2\ge 0 }\exp\left(\theta (t_1-b_3)-{(t_1-b_3)^2+t_2^2\over 2}\right)dt_1dt_2.
\end{align*}
}

\section{Solving the optimization problem (\ref{minimax})}\label{section:SystemEquation}
Given $x$ satisfying (\ref{eq:xupperrange}) and $\beta \in (0,1/2)$, we recognize the optimization problem in (\ref{minimax}) is a Minimax problem. In this section, we will rely on Sion's Minimax Theorem~\cite[Corollary 3.3]{sion1958general} to solve it. We first use the degree local optimality constraint to claim that we only need to consider a bounded set of $t$. Recall from (\ref{z_ij_cons})
\begin{align*}
z_{1,4}-2z_{1,1} \geq |z_{1,2}-z_{1,3}|, \quad z_{2,3}-2z_{2,2} \geq |z_{1,2}-z_{2,4}|,
\end{align*}
which by (\ref{eq:eta}) gives that
\begin{align}\label{dec_1}
\eta_{1,4}-2\eta_{1,1} \geq |\eta_{1,3}-\eta_{1,2}|, \quad \eta_{2,3}-2\eta_{2,2} \geq |\eta_{1,2}-\eta_{2,4}|.
\end{align}
Recall that $t=\eta_{1,4}-2\eta_{1,1}$ and $x-t=\eta_{2,3}-2\eta_{2,2}$ and from (\ref{dec_1}), we have that $t \in [0,x]$. For a given $x$ and $\beta$, we rewrite the minimax problem in (\ref{minimax}) as
\begin{align*}
\sup_{t\in[0,x]} \inf_{\theta_1, \theta_2} F(t,\theta_1,\theta_2)
\end{align*}
where
\begin{align}\label{F_t_theta}
F(t,\theta_1,\theta_2)= &-\frac{t^2}{2\beta^2}-\frac{(x-t)^2}{2(1/2-\beta)^2}+2\beta \log P\left( \theta_1,\sqrt{\frac{1/2-\beta}{\beta}},\frac{t}{\beta^{3/2}}\right) \nonumber \\
                        & +2(1/2-\beta)\log P\left( \theta_2,\sqrt{\frac{\beta}{1/2-\beta}},\frac{x-t}{(1/2-\beta)^{3/2}} \right)
\end{align}
By the convexity of $\log M(\theta)$ and the concavity of $\log P(\theta_1,b_2,b_3)$ in $b_3$ which was established in Lemma \ref{lemma:theta1b1}, we have that $F(t,\cdot,\cdot)$ is convex on a set $(\theta_1,\theta_2)\in \mathbb{R}^2$, and $F(\cdot,\theta_1,\theta_2)$ is concave on $[0,x]$. Sion's Minimax Theorem then gives
\begin{align}
\label{saddle_point}
\sup_{ t \in [0,x] } \inf_{\theta_1,\theta_2} F(t,\theta_1,\theta_2)= \inf_{\theta_1,\theta_2 } \sup_{ t \in [0,x] } F(t,\theta_1,\theta_2).
\end{align}
Given $x$ and $\beta$, let the saddle point set be $t^{*}(x,\beta) \times (\theta_1^{*}(x,\beta), \theta_2^{*}(x,\beta)) \subset [0,x] \times \mathbb{R}^2$, where
\begin{align}
t^*=t^{*}(x,\beta)=&\text{argmax}_{t \in [0,x]} \inf_{\theta_1,\theta_2} F(t,\theta_1,\theta_2),   \nonumber \\
\theta^*=(\theta_1^{*}(x,\beta), \theta_2^{*}(x,\beta))=&\text{argmin}_{(\theta_1,\theta_2) }  \sup_{t \in [0,x]} F(t,\theta_1,\theta_2). \nonumber
\end{align}
\begin{lemma}\label{lemma:Uniquesaddlepoint}
Given any $x$ satisfying (\ref{eq:xupperrange}) and any $\beta \in (0,1/2)$, $(t^*,\theta^*)$ is unique and is given as the unique solution to
\begin{align}
\label{Fpartialthetay}
\frac{\partial F(t,\theta_1,\theta_2)}{\partial t}=0, \; \frac{\partial F(t,\theta_1,\theta_2)}{\partial \theta_1}=0, \; \frac{\partial F(t,\theta_1,\theta_2)}{\partial \theta_2}=0.
\end{align}
\end{lemma}
\begin{proof}
Let $G(t)=\inf_{\theta_1,\theta_2} F(t,\theta_1,\theta_2)$. By Lemma \ref{lemma:theta1b1}
$G(t)$ is strictly concave in $t$. For any $\epsilon>0$, we claim that
\begin{align} \label{G_claim}
G (-\epsilon )=G (x+\epsilon )=-\infty.
\end{align}
Recall $P(\theta_1,b_2,b_3)=P(\theta_1,0,0,b_2,b_3)$ and from (\ref{eq:Pfunction}), we have
\begin{align}\label{P_eps}
\log P \left(\theta_1,\sqrt{\frac{1/2-\beta}{\beta}},-\frac{\epsilon}{\beta^{3/2}} \right)&=\log\int_{t_1 \geq \sqrt{\frac{1/2-\beta}{\beta}}|t_2|+\frac{\epsilon}{\beta^{3/2}}} \exp(\theta_1 t_1)d\mu(t_1,t_2)
\end{align}
where $\mu(\cdot)$ is the probability measure induced by two i.i.d. standard normal random variables. Let the domain of the integration above be
$$
\mathcal{D}=\left\{  (t_1,t_2): t_1 \geq \sqrt{\frac{1/2-\beta}{\beta}}|t_2|+\frac{\epsilon}{\beta^{3/2}} \right\}
$$
For $\theta_1<0$, we have (\ref{P_eps})
\begin{align*}
\leq \log\left( \exp \left(\theta_1 \frac{\epsilon}{\beta^{3/2}} \right) \mu(\mathcal{D})\right)=\theta_1 \frac{\epsilon}{\beta^{3/2}} + \log \mu(\mathcal{D})
\end{align*}
As $\theta_1 \rightarrow -\infty$, we have the right hand side of the equation above goes to $-\infty$ and hence
$$
\inf_{\theta_1}\log P \left(\theta_1,\sqrt{\frac{1/2-\beta}{\beta}},-\frac{\epsilon}{\beta^{3/2}} \right)=-\infty,
$$
which from (\ref{F_t_theta}) implies $G(-\epsilon)=-\infty$. For $t=x+\epsilon$, applying the same argument to another part in (\ref{F_t_theta}) yields $G(x+\epsilon)=-\infty$. This establishes the claim (\ref{G_claim}). Thus $\sup_{t\in [0,x]} G(t)$ is achieved by a unique $t=t^*\in [0,x]$. Lemma \ref{lemma:theta1b1} yields that $(\theta_1^*,\theta_2^*)$ which obtains the infimum of $F(t^*,\theta_1,\theta_2)$ is the unique solution to the last two equations in (\ref{Fpartialthetay}) for $t=t^*$. Fix $(\theta_1,\theta_2)=(\theta_1^*, \theta_2^*)$.
The strict concavity of $F(t,\theta_1^*,\theta_2^*)$ in $t$ and the maximality of $t^*$ indicates that $t^*$ is the unique solution to the first equation in (\ref{Fpartialthetay}). Hence, we have $(t^*,\theta_1^*, \theta_2^*)$ is a solution to (\ref{Fpartialthetay}) as claimed.

The concavity of $F(t,\theta_1,\theta_2)$ in $t$, $\forall (\theta_1,\theta_2)\in \mathbb{R}^2$, yields
\begin{align}
\label{sad_con1}
& F(t,\theta_1^*,\theta_2^*) \leq F(t^*,\theta_1^*,\theta_2^*)+\frac{\partial F(t,\theta_1^*,\theta_2^*)}{\partial t}\bigg|_{t=t^*}(t-t^*) \nonumber \\
& \Rightarrow F(t,\theta_1^*,\theta_2^*) \leq F(t^*,\theta_1^*,\theta_2^*) , \quad \forall t\in [0,x]
\end{align}
Similarly, the convexity of $F(t,\theta_1,\theta_2)$ in $(\theta_1,\theta_2)$, $\forall t\in[0,x]$, yields
\begin{align}
\label{sad_con2}
F(t^*,\theta_1,\theta_2) \geq F(t^*,\theta_1^*,\theta_2^*), \quad \forall (\theta_1,\theta_2)\in \mathbb{R}^2
\end{align}
(\ref{sad_con1}) and (\ref{sad_con2}) hence implies that $(t^{*}, \theta_1^*, \theta_2^*)$ is the saddle point of $F(t,\theta_1,\theta_2)$. Next, we show this saddle point is unique. Suppose there is another saddle point $(\hat{t}, \hat{\theta}_1, \hat{\theta}_2) \neq (t^{*}, \theta_1^*, \theta_2^*)$. If $(\hat{\theta}_1, \hat{\theta}_2) \neq (\theta_2^*, \bar{\theta}_2^*)$, the strict convexity of $F(t,\theta_1,\theta_2)$ in $(\theta_1,\theta_2)$, $\forall y\in [0,x]$, gives
\begin{align}
\label{sad_con3}
F(t^*,\theta_1^*,\theta_2^*) < F(t^*,\hat{\theta}_1,\hat{\theta}_2)
\end{align}
while the saddle point property of $(\hat{t}, \hat{\theta}_1, \hat{\theta}_2)$ implies
\begin{align}
\label{sad_con4}
F(t^*,\hat{\theta}_1,\hat{\theta}_2) \leq F(\hat{t},\hat{\theta}_1,\hat{\theta}_2)
\end{align}
Then from (\ref{sad_con3}) and (\ref{sad_con4}), we have $F(t^*,\theta_1^*,\theta_2^*) < F(\hat{t},\hat{\theta}_2,\hat{\theta}_2)$ which is a contradiction. Likewise if $\hat{t}\neq t^*$, we can use the strict concavity of $F(t,\theta_2,\bar{\theta}_2)$ in $t$, $\forall (\theta_1,\theta_2) \in \mathbb{R}^2$, and the saddle point property of $(\hat{t}, \hat{\theta}_1, \hat{\theta}_2)$ to construct a contradition. Hence the uniqueness of $(t^*,\theta_1^*,\theta_2^*)$ as a saddle point follows, which also implies
that the solution to (\ref{Fpartialthetay}) is unique.
\end{proof}
Next, we derive the explicit expressions for the partial derivatives in (\ref{Fpartialthetay}), which are (\ref{partF_the1}), (\ref{partF_the2}) and (\ref{partF_t}), respectively. Hence, Lemma \ref{lemma:UniqueSolution2} follows from Lemma \ref{lemma:Uniquesaddlepoint}. From (\ref{unique_theta1}) in Lemma \ref{lemma:theta1b1}, we have
\begin{align*}
&\frac{\partial F(t,\theta_1,\theta_2)}{\partial \theta_1}=0 \Rightarrow \text{(\ref{partF_the1})}, \quad \frac{\partial F(t,\theta_1,\theta_2)}{\partial \theta_2}=0 \Rightarrow  \text{(\ref{partF_the2})}.
\end{align*}
Next, we have
\begin{align*}
\frac{\partial \log P(\theta,b_2,b_3)}{\partial t}=& \frac{\partial \log P(\theta,b_2,b_3)}{\partial b_3} \frac{\partial b_3}{\partial t},
\end{align*}
Since $t$ appears only in $b_3$. From (\ref{P_002}),
\begin{align}
 &\frac{\partial (\log (\exp(\theta_1^2/2)/\pi)+\log Q(\theta,b_2,b_3))}{\partial b_3} \frac{\partial b_3}{\partial t} \nonumber \\
\label{partialQ_theta}
&=  \frac{\partial \log Q(\theta,b_2,b_3)}{\partial b_3} \frac{\partial b_3}{\partial t}
\end{align}
By the expression of $Q(\theta,b_2,b_3)$ in (\ref{Q_exp}), it is easy to see that
$$
\frac{\partial \log Q(\theta,b_2,b_3)}{\partial b_3}=\frac{\partial \log Q(\theta,b_2,b_3)}{\partial \theta}.
$$
From (\ref{unique_theta1}), we have
$$
\frac{\partial \log Q(\theta,b_2,b_3)}{\partial \theta}=\frac{1}{Q(\theta,b_2,b_3)} \frac{\partial Q(\theta,b_2,b_3)}{\partial \theta}=-\theta.
$$
Then (\ref{partialQ_theta}) becomes $-\theta \frac{\partial b_3}{\partial t}$. Hence, we have
\begin{align*}
\frac{\partial F(t,\theta_1,\theta_2)}{\partial t}=0 \Rightarrow \text{(\ref{partF_t})}.
\end{align*}
Next we rely on a numerical approach to finding $x_l$ defined in (\ref{lower_bound_x}). First we claim that $W(x,1/4)=2w(x)$ for any $x$ satisfying (\ref{eq:xupperrange}). For $\beta=1/4$, it is easy to see from (\ref{partF_the1}) and (\ref{partF_the2}) that $\theta_1=\theta_2$ and then $t=x/2$ follows from (\ref{partF_t}). Hence, it is the same computation scenario as the special setting in the beginning of subsection \ref{subsection:optimizingI1I2}, then
\begin{align*}
W(x,1/4)=&2\log 2 -4x^2 + \log P(2\theta(x),1,4x)
\end{align*}
where $\theta(x)$ is the unique solution to (\ref{transce2}) for a given $x$ satisfying (\ref{eq:xupperrange}). Recall a form of $P(2\theta(x),1,4x)$ given in (\ref{P_002}) and by Lemma \ref{basis_cal}, the last equation becomes
\begin{align*}
W(x,1/4)=& 2\log2-4x^2 + 2\theta^2(x) + \log \left( \frac{1}{\pi} Q(2\theta(x),1,4x) \right) \\
        =& -4x^2 + 2\theta^2(x) + 2\log (1+\erf(\theta(x)+2x)) \\
				=& 2w(x),
\end{align*}
as claimed. From the expression of $W(x,\beta)$ in (\ref{eq:Wxbeta}), it is easy to see that $W(x,\beta)$ is symmetric about $\beta=1/4$. For the maximum of $W(x,\beta)$ over $\beta\in(0,1/2)$, we only need to consider the region $(0,1/4]$.

Let
\begin{align*}
L(x,\beta,t,\theta_1,\theta_2)=&-2\beta \log \beta-2(1/2-\beta)\log(1/2-\beta)-\frac{t^2}{2\beta^2}-\frac{(x-t)^2}{2(1/2-\beta)^2} \\
                        & +2\beta \log P\left( \theta_1,\sqrt{\frac{1/2-\beta}{\beta}},\frac{t}{\beta^{3/2}}\right)+2(1/2-\beta)\log P\left( \theta_2,\sqrt{\frac{\beta}{1/2-\beta}},\frac{x-t}{(1/2-\beta)^{3/2}} \right).
\end{align*}
Finally, we numerically compute $x_l$ in (\ref{lower_bound_x}) based on the bisection method, in which for a given $x$ satisfying (\ref{eq:xupperrange}) we use the command `FindRoot' in Mathematica to search for a solution to the equation system (\ref{partF_the1}), (\ref{partF_the2}), (\ref{partF_t}) and $L(x,\beta,t,\theta_1,\theta_2)=2w(x)$ inside the region $\beta \in [10^{-10}, 1/4-10^{-10}]$. If the search succeeds, we set $x$ as an upper bound of $x_l$, otherwise we set $x$ as a lower bound of $x_l$. The numerical search procedure using the above choice of parameters converges to $x = 0.47523..$~. Assuming the validity of the numerical search, the result $x_l = 0.47523..$ follows. We plot below the functions $W(x,\beta)$ for $x=0.47523$ and $x=0.5$.

\ignore{
\\
\begin{algorithmic}[]
\STATE Set $tol$ and $\epsilon$ to be two small positive number, and $a$ and $b$ to be the known lower and upper bound of $x$.
\WHILE {$(b-a)>tol$}
        \STATE  $x=a+(b-a)/2$;
				\STATE  Solve (\ref{transce2}) and get a unique solution $\theta(x)$;
				\STATE Search for a solution to the equation system (\ref{partF_the1}), (\ref{partF_the2}), (\ref{partF_t}) and $L(x,\beta,t,\theta_1,\theta_2)=2w(x)$ inside the region $\beta \in [\epsilon, 1/4-\epsilon]$. If the search succeeds, let $ind=1$, otherwise let $ind=0$;
				\IF {$ind=1$}
				     \STATE  $b=x$;
				\ELSE
						 \STATE  $a=x$;
				\ENDIF
\ENDWHILE
\STATE Output $x_l=a$;
\end{algorithmic}
\bigskip
}

\begin{figure}[!htb]
  \centering
  \includegraphics[width=4in]{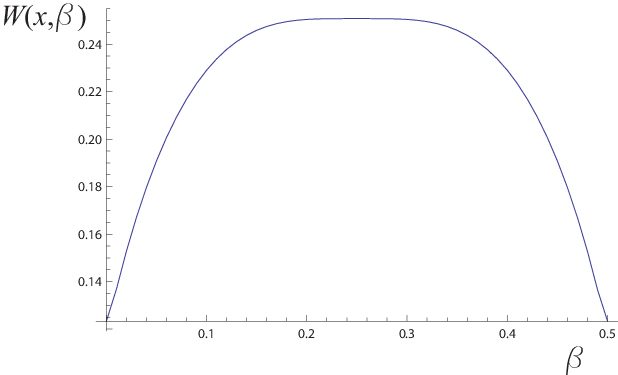}
  \caption[]
   {$W(x,\beta)$ for $x=0.47523$}
	\label{SMMM}
\end{figure}
\begin{figure}[!htb]
  \centering
  \includegraphics[width=4in]{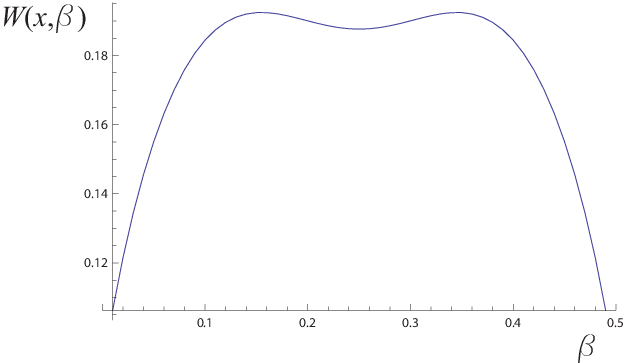}
  \caption[]
   {$W(x,\beta)$ for $x=0.5$}
	\label{SMMM2}
\end{figure}

\section{Proof of Theorem \ref{cgraph_cut} } \label{section:proof_MaxCut_cubic}
We note that the proof below does not rely on any of the ideas developed in the earlier section and relies on a completely different approach. Specifically,
to construct a cut on random cubic graph or cubic graph with large girth, we make use of the following theorem on induced bi-partite subgraphs with a lot of vertices as a starting point.
\begin{theorem}
{\upshape \cite[Theorem 2]{csoka2014invariant}}
Every cubic regular graph with sufficiently large girth has an induced subgraph that is bi-partite and that contains at least a $0.86$ fraction of the vertices.
\end{theorem}
It implies that besides a bipartite subgraph, there are at most $0.14n$ vertices outside the bi-partite subgraph. As a result, we have three separate vertex subsets, two in the bipartite subgraph and one consisting of the vertices outside of the bipartite subgraph. Firstly, we color the two separate vertex subsets in the bipartite subgraph with $0$ and $1$, respectively. Then we color the remaining at most $0.14n$ vertices one by one. Choose one vertex u.a.r. among all the uncolored vertices which have the largest number of edges connecting to the colored vertices, and then color the vertex oppositely to the majority color of its colored neighbors. If the selected vertex has equal number of neighbors of different colors, randomly color this vertex. Since the graph is connected, this coloring procedure will not be terminated until all the vertices are colored. Since coloring one vertex brings at most one edge with both ends inside one vertex subset of the same color, this coloring procedure produces a large cut with cut size at least $1.5n-0.14n=1.36n$, which gives Theorem \ref{cgraph_cut}.

\section{Conclusions and further questions}\label{section:OpenProblems}
There are several questions which remain unanswered after our work. First it would be nice to tighten the result and obtain matching upper and lower bounds on the coefficient of $\sqrt{c}$ in the upper and lower bounds on the Max-Cut value.
For that matter we do not even know whether this quantity is well defined and thus leave it as a challenge to first establish the existence
of the limit
\begin{align}\label{eq:MaxCutLimitSecond}
x^*=\lim_{c\rightarrow\infty}{\mathcal{MC}(c)-c/2\over \sqrt{c}}
\end{align}
and second, identifying the value of $x^*$. It is worth noting that the method that was introduced recently to address the existence of such limits in similar contexts,
namely the interpolation method~\cite{BayatiGamarnikTetali}, and which was used to make the quantity $\mathcal{MC}(c)$ a well-defined value, does not seem to work here.
Thus our first open question is:
\begin{OpenProblem}
\label{openproblem1}
Establish the existence of the limit (\ref{eq:MaxCutLimitSecond}) and identify the value of this limit.
\end{OpenProblem}
{\remark Dembo, Montanari and Sen \cite{Dembo2015} resolved this question positively and computed the limit
(\ref{eq:MaxCutLimitSecond}). The limit was shown to be related to the ground state of the Sherrington-Kirkpatrick model.}

Our second group of questions relates to the concept of
\emph{i.i.d. factors} which appear in the context
of theory of converging sparse graphs~\cite{HatamiLovaszSzegedy},\cite{lyons2011perfect},\cite{gamarnik2014limits},\cite{rahman2014local},\cite{csoka2014invariant}.
The concept appears also under name \emph{coding invariant processes} in Open Problem 2.0 in~\cite{Aldous:FavoriteProblems}.
We do not formally define here i.i.d. factors as it falls somewhat out of the scope of the paper, and instead refer the reader to the literature above. One of the outstanding
questions in this area is identifying the largest density obtainable on infinite trees with a fixed degree distribution, for example a regular (Kelly) tree.
It was shown in~\cite{gamarnik2014limits} and later in~\cite{rahman2014local} that the clustering property provides upper bound on the density of i.i.d. factors.
This approach applies to the case of Max-Cut value as well. Specifically, let $\T_{\Pois}$ denote a (finite or infinite)
tree obtained as a Galton-Watson process with Poisson off-spring distribution with parameter $c$.
As an implication of the upper bound part of our main result we obtain
\begin{coro}\label{coro:MaxCutFactorIID}
The largest Max-Cut density on $\T_{\Pois}$ obtainable as a factor of i.i.d. is at most $\mathcal{MC}(c/2)$.
\end{coro}
Here the argument of $\mathcal{MC}$ is $c/2$ instead of $c$ is due to the fact that the average degree in $\Gncn$ graph is $c/2$.
The proof of this result follows from the argument very similar to the one found in~\cite{gamarnik2014limits}. Nevertheless, since the clustering property
is not yet established for the Max-Cut problem it is not clear whether $\mathcal{MC}(c/2)$ is achievable as a factor of i.i.d.
Our last open problem concerns this question.
\begin{OpenProblem}
Determine whether the value $\mathcal{MC}(c/2)$ is achievable as factor of i.i.d. process.
\end{OpenProblem}

\bibliographystyle{amsalpha}
\bibliography{bibliography}

\end{document}